\documentclass[psamsfonts]{amsart}
\usepackage{cite,mathrsfs,amsmath,amssymb,amsthm,bbm}
\usepackage[english]{babel}
\usepackage[T1]{fontenc}
\usepackage[utf8]{inputenc}
\usepackage{tikz}
\usepackage{todonotes}
\usetikzlibrary{decorations.markings}
\usepackage{verbatim}

\usepackage[in,headings]{fullpage}
\usepackage[osf,theoremfont,largesc]{newtxtext}
\usepackage{microtype}

\usepackage[hang,flushmargin,bottom]{footmisc}

\usepackage[colorlinks=true,urlcolor=magenta,citecolor=magenta,linkcolor=magenta,breaklinks=true]{hyperref}

\newtheorem{lem}{Lemma}[section]
\newtheorem{thrm}[lem]{Theorem}
\newtheorem{prop}[lem]{Proposition}
\newtheorem{cor}[lem]{Corollary}
\theoremstyle{definition}
\newtheorem{defn}[lem]{Definition}

\theoremstyle{remark}
\newtheorem{rem}[lem]{Remark}

\newcommand{\eq}[2]{\begin{equation}\label{#1}#2\end{equation}}

\newcommand{\R}{\mathbb{R}}
\newcommand{\C}{\mathbb{C}}
\newcommand{\Z}{\mathbb{Z}}
\newcommand{\N}{\mathbb{N}}
\newcommand{\T}{\mathbb{T}}

\newcommand{\bbS}{\mathbb{S}}

\newcommand{\M}{\mathcal{M}}
\newcommand{\Schwartz}{\mathcal{S}}

\renewcommand{\epsilon}{\varepsilon}

\newcommand{\mr}{\mathrm}

\newcommand{\<}{\langle}
\renewcommand{\>}{\rangle}
\newcommand{\p}{\partial}

\newcommand{\bbo}{\mathbbm 1}

\newcommand{\bpf}{\begin{proof}}
\newcommand{\epf}{\end{proof}}
\newcommand{\qtq}[1]{\quad\text{#1}\quad}
\newcommand{\qt}[1]{\quad\text{#1}}

\usepackage{enumerate}

\newcommand{\nonlineary}{\delta_{\{y=0\}}}

\numberwithin{equation}{section}

\title[WELL-POSEDNESS for Cubic NLS with Nonlinearity Supported on a Hypersurface]{WELL-POSEDNESS for the Cubic NONLINEAR SCHR\"ODINGER EQUATION With Nonlinearity Supported on a Hypersurface}
\date{}
\author{Xilin Chen}
\address{Department of Mathematics \& Statistics, Georgetown University\\Washington, DC 20057, USA}
\email{xc243@georgetown.edu}
\thanks{}

\author{Jianlei Huang}
\address{Department of Mathematics \& Statistics, Georgetown University\\Washington, DC 20057, USA}
\email{jh2524@georgetown.edu}
\thanks{}

\begin{document}

\begin{abstract}
We prove that the 2-d cubic nonlinear Schr\"odinger equation with nonlinearity supported on a line or a sphere is globally well-posed in \(H^1(\R^2)\). For higher dimensions, we prove that the equation with hyperplane-supported nonlinearity is locally well-posed in suitable anisotropic Sobolev spaces. We also prove that the equation with sphere-supported nonlinearity is globally well-posed in \(H^1(\R^n)\) for radial data.
\end{abstract}

\maketitle

\section{Introduction}

We consider the dynamics of a field \(\psi:\R^n\to\C\) under the Hamiltonian
\begin{equation}\label{E}
    E[\psi]:=\frac 1 2\int_{\R^n}|\nabla\psi(x)|^2\,dx+\frac \mu 2\int_\Sigma |\psi(x)|^4 \,dS(x),
\end{equation}
where \(\mu=\pm1\) and \(\Sigma\subseteq \R^n\) is a smooth hypersurface with surface measure \(S\). This leads to the cubic nonlinear Schr\"odinger equation with nonlinearity supported on \(\Sigma\),
\begin{equation}\label{nls-sigma}\tag{{NLS}$_\Sigma$}{
\begin{aligned}
 i\partial_t \psi + \Delta \psi &= 2\mu\delta_\Sigma\,|\psi|^2\psi.
\end{aligned}}
\end{equation}
Here, \(\delta_\Sigma\) is understood in the sense of distributions, i.e.,
\[
\langle \delta_{\Sigma},f\rangle=\int_{\Sigma}f\ dS
\]
 for all \(f\in C_c^\infty(\R^n)\). We say \eqref{nls-sigma} is "defocusing" if \(\mu=1\), and "focusing" if \(\mu=-1\).

From a physical perspective, this model describes a medium in which the nonlinear response is negligible in the bulk but becomes significant along a thin interface or defect. Studying such an idealized system helps isolate how a spatially localized nonlinear interaction modifies dispersive propagation. 
From a modeling perspective, \eqref{nls-sigma} may be viewed as an idealized dispersive system whose cubic response is localized on a codimension-one defect. Away from the interface, the field undergoes free Schrödinger evolution; formally, the effect of the defect can instead be encoded through continuity of the field across \(\Sigma\) together with a nonlinear jump condition for its normal derivative \cite{Finco_2023}.
This localization leads to an analytical structure substantially different from that of the standard nonlinear Schrödinger equation. The nonlinearity acts through the trace of the solution on \(\Sigma\), rather than through its values in the ambient space. The central difficulty is therefore to control the interaction between Schrödinger dispersion, trace regularity, and the resulting measure-valued forcing term.

Nonlinear Schr\"odinger equations with concentrated nonlinearities have been studied in several settings. For a recent survey of results on the point-concentrated nonlinearity, see the review of Tentarelli \cite{Tentarelli_review}. In one dimension, where $\Sigma = \{x=0\} $, global well-posedness in the defocusing case for initial data in \(H^1(\R)\) was proved in \cite{Adami2001,Komech2007}. A threshold theorem was proved in the focusing case by \cite{Holmer2020}. The equation \eqref{nls-sigma} is mass-critical in one dimension. Local well-posedness in $L^2(\R)$ was proved in \cite{Adami2021}, with global well-posedness and scattering for small data obtained in \cite{Adami2021, Hayashi2023}. Large data scattering in $L^2(\R)$ was recently proved by Harrop-Griffiths, Killip, and Vi\c{s}an \cite{harropgriffiths2025scatteringnonlinearschrodingerequation}. Rather less is known in higher dimensions. Well-posedness for point-concentrated nonlinearities has been investigated in two dimensions by Carlone, Correggi, and Tentarelli \cite{Carlone_2019}, and in three dimensions by Adami, Dell'Antonio, Figari, and Teta \cite{Adami2003}. In dimensions greater than or equal to two, the point-concentrated problem differs significantly from the hypersurface-concentrated problem considered here. Crucially, point evaluation is not controlled by the $H^1(\mathbb R^2)$ norm, whereas restriction to a smooth curve can be treated using trace estimates. 

To the best of the authors' knowledge, the only result that directly addresses the model \eqref{nls-sigma} when $n \geq 2$ is by Finco, Tentarelli, and Teta \cite{Finco_2023}, who prove global well-posedness with nonlinearities concentrated on $\Sigma = \bbS^2\subset\mathbb R^3$ for suitable initial data.  Our result, Theorem \ref{thrm:radial}, complements this work by proving global well-posedness in the energy space for radial initial data. A related problem when $\Sigma $ is a hyperplane is the nonlinear Sch\"odinger equation on half-spaces with nonlinear boundary conditions. Well-posedness in the mass-critical and mass-subcritical settings as well as Strichartz estimates have been recently studied in \cite{garofalo2026strichartzestimatesschrodingerequations,Hayashi2021}.

In this paper, we focus on two canonical hypersurfaces in $\mathbb{R}^n$: the hyperplane and the unit sphere. In the hyperplane case, let \(x=(x',y)\in\R^{n-1}\times\R=\R^n\), and let \(\Sigma=\{y=0\}\) so that \eqref{nls-sigma} becomes
\begin{equation}\label{nls-y}\tag{{NLS}$_y$}{
\begin{aligned}
 i\partial_t \psi + \Delta \psi &= 2\mu\delta_{\{y=0\}}\,|\psi|^2\psi.
\end{aligned}}
\end{equation}In the unit sphere case, let \(\Sigma=\mathbb S^{n-1}=\{|x|=1:x\in \R^n\}\), which is the \((n-1)-\)dimensional sphere in \(\R^n\), so that \eqref{nls-sigma} becomes
\begin{equation}\label{nls-S}\tag{{NLS}$_{\mathbb S}$}{
\begin{aligned}
 i\partial_t \psi + \Delta \psi &= 2\mu\delta_{\mathbb S^{n-1}}\,|\psi|^2\psi.
\end{aligned}}
\end{equation}

Scalling symmetry frequently indicates in which space we might hope to have well-posedness. While the general model \eqref{nls-sigma} does not possess a scaling symmetry, the hyperplane case \eqref{nls-y} does enjoy this property: If \(\psi(t,x)\) is a solution to \eqref{nls-y}, then
\[
\psi_\lambda:=\sqrt{\lambda}\psi(\lambda^2t,\lambda x)
\]
is a  solution of \eqref{nls-y} with initial data \(\psi_{0,\lambda}=\sqrt{\lambda}\psi_0(\lambda x)\). The corresponding scalling critical sobolev space is \(\dot{H}^{\frac{n+1}{2}}(\R^n)\). Crucially, \eqref{nls-y} is energy sub-critical when \(n=2\), critical when $n = 3$, and super-critical when \(n\geq 4\).

Our first result, Theorem \ref{thrm:GWPy}, proves $H^1$ well-posedness of \eqref{nls-y} and \eqref{nls-S} in the energy-subcritical case \(n=2\). Our second result, Theorem \ref{thrm:LWP}, considers \eqref{nls-y} in the energy-critical and energy-supercritical cases by working  with solutions in anisotropic Sobolev spaces with sufficient smoothness. As \eqref{nls-S} does not enjoy scaling symmetry, our final and possibly most surprising result is that $H^1$ well-posedness holds for \eqref{nls-S} for all $n\geq 3$ with radial data.

Before stating our main results, let us first make the following definition:

\begin{defn}\label{def:duhm}
Let \(I\subseteq\R\) be a (time) interval containing \(0\), and \(X\) be a subspace of \(L^2(\R^n)\). We say \(\psi\in C(I;X)\) is a strong solution to \eqref{nls-sigma} with initial data \(\psi_0\) if \(\psi\in L^4_{\mr{loc}}(I;L^4(\Sigma,dS))\) and for all \(t\in I\) we have
    \begin{equation}\label{eq:Dulm}
    \psi(t)=e^{it\Delta}\psi_0-i2\mu\int_{0}^te^{i(t-s)\Delta}[\delta_{\Sigma}|\psi(s)|^2\psi(s)]ds
\end{equation}
\end{defn}

Our main results comprise the following three theorems:
\begin{thrm}\label{thrm:GWPy}
Let  \(n=2\), there exists a unique strong global solution \(\psi\in C(\R;\,H^1(\R^2))\) to both \eqref{nls-y} and \eqref{nls-S} with initial data $\psi_0\in H^1(\R^2)$ if either of the following two conditions holds:

(i)  \eqref{nls-sigma} is defocusing;

(ii) \(\|\psi_0\|_{H^1(\R^2)}\) is sufficiently small.

This solution conserves both the mass and the energy. Moreover, for any \(T>0\), the solution map \(H^1(\R^2)\ni \psi_0\mapsto \psi\in C([-T,T];H^1(\R^2))\) is continuous. 
\end{thrm}

\begin{thrm}\label{thrm:LWP}
    Let \(n\geq 2\) and \(s>\frac {n-1} 2\), then given any initial data \(\psi_0\in H_{x'}^sL_{y}^2\), there exists a time \(T=T(\|\psi_0\|_{H_{x'}^sL_{y}^2})>0\) and a unique strong solution \(\psi\in C([-T,T];H_{x'}^sL_{y}^2)\cap L^4([-T,T];H_{x'}^sC_{0,y})\) of \eqref{nls-y} satisfying \(\psi(0)=\psi_0\). This solution conserves the mass. Moreover, the data-to-solution map \(H_{x'}^sL_{y}^2\ni \psi_0\mapsto \psi\in C([-T,T];H_{x'}^sL_{y}^2)\) is Lipchitz continuous.
\end{thrm}

\begin{thrm}\label{thrm:radial}
Let  \(n\geq 3\), there exists a unique strong global radial solution \(\psi=\psi(r)\in C(\R;\,H^1(\R^n))\) to \eqref{nls-S} with radial initial data $\psi_0\in H^1(\R^n)$ if either of the following two conditions holds:

(i)  \eqref{nls-S} is defocusing;

(ii) \(\|\psi_0\|_{H^1(\R^n)}\) is sufficiently small.

This solution conserves both the mass and the energy. Moreover, for any \(T>0\), the solution map \(H^1(\R^n)\ni \psi_0\mapsto \psi\in C([-T,T];H^1(\R^2))\) is continuous. 
\end{thrm}

\begin{rem}
In each of Theorems~\ref{thrm:GWPy}--\ref{thrm:radial} we may replace \(\delta_\Sigma\) by the weighted measure \(\delta_{\Sigma,w}\), defined by
 \[
\langle \delta_{\Sigma,w},f\rangle:=\int_{\Sigma}wf\,dS,
\]
for all \(f\in C_c^\infty(\R^n)\), where \(w\in L^\infty(\Sigma,dS)\) is real-valued. In particular, we may take \(w = \bbo_E\) to be an indicator function on a Borel subset \(E\subseteq \Sigma\), where \(\Sigma\) is respectively either a hyperplane or a sphere. Note that if \(E\) consists of a single point then \(\delta_{\Sigma,w} = 0\) so we do not recover the point-concentrated model, but rather the linear Schr\"odinger equation.
\end{rem}

Our proof of Theorem~\ref{thrm:GWPy} follows a regularization strategy
similar to that used for point-concentrated nonlinearities in
\cite{Komech2007}, with an implementation inspired by the treatment of the sprinkled nonlinearity in \cite{Harrop-Griffiths_2025}. We construct solutions to \eqref{nls-sigma} as limits of
solutions to the regularized problem
\[
    i\partial_t\psi^\epsilon+\Delta\psi^\epsilon
    =2\mu\rho_\Sigma^\epsilon|\psi^\epsilon|^2\psi^\epsilon,
    \qquad
    \psi^\epsilon(0)=\psi_0,
\]
where $\rho_\Sigma^\epsilon$ is a smooth approximation of
$\delta_\Sigma$. In the defocusing case, conservation of mass and energy yield $H^1$
bounds for the approximate solutions that are uniform in $\epsilon$; in the
focusing case, the corresponding bounds follow from a small-data bootstrap argument.

We first establish convergence of the approximate solutions in a weaker topology. After proving conservation of mass and energy for the limiting solution, which relies on a uniqueness argument obtained from Gronwall's inequality, we obtain convergence of the mass and energy and consequently the $H^1$-norm. We then apply the Radon--Riesz argument to obtain convergence of the approximate solutions in $H^1$.

The estimates needed to implement this strategy depend on the geometry of $\Sigma$. Even for the line $\Sigma=\{y=0\}$, the argument requires several technical innovations, including anisotropic dispersive and Sobolev estimates adapted to the
tangential and normal variables. These estimates provide sufficient control of the concentrated nonlinearity to both pass to the limit in our approximating sequence and close the Gronwall inequality argument. Moreover, in contrast to
\cite{Komech2007,Harrop-Griffiths_2025},
we employ this argument to treat the focusing case for sufficiently small initial data in addition to the defocusing case.

The passage from the line to the circle $\Sigma=\bbS^1$ is significantly more challenging. We again work in anisotropic spaces, but now in the radial and angular variables $(r,\theta)$ as opposed to the Cartesian variables $(x,y)$.
The key obstacle to proving convergence of the approximate solutions is the $|t|^{-1}$ decay in the standard two-dimensional dispersive estimate, which is not integrable near zero. To overcome this difficulty we prove a weighted dispersive estimate that, to our knowledge, is new; see Lemma~\ref{dispersive-S-C}. For sufficiently small $\epsilon$, the regularized nonlinearities are supported in a fixed spatial annulus
bounded away from the origin and infinity. Combining the weighted estimate with this spatial localization gives us an integrable rate of decay, which in turn allows us to pass to the limit in our approximate solutions.

Our new weighted dispersive estimate in the case \(n=2\) can be generalized to higher dimensions (Lemma~\ref{dispersive_r_high}); this enables us to treat the higher-dimensional radial case in
Theorem~\ref{thrm:radial}. To use this new dispersive estimate, we once again rely on the the annular support of the nonlinearity. Radial symmetry supplies the other essential ingredient: on a fixed annulus bounded away from the origin, the $H^1(\mathbb R^n)$-norm controls the one-dimensional $H^1$-norm of the radial profile and hence its supremum. Thus radial symmetry and localization recover the pointwise control required on the support of the nonlinearity, allowing the regularization argument to extend
to higher dimensions.

Finally, the local well-posedness result in
Theorem~\ref{thrm:LWP} follows by a direct contraction mapping
argument. The new ingredient here is an anisotropic Strichartz
estimate (Lemma~\ref{thrm:stri}), which provides the space-time
bounds needed to close the contraction.

The remainder of this article is organized as follows.
Section~2 collects the preliminary estimates used repeatedly
throughout the paper.
Sections~3 and~4 establish global \(H^1\) well-posedness in two
dimensions for the line and circle cases, respectively, using
estimates adapted to each geometry.
Section~5 proves local well-posedness for the hyperplane problem
in dimensions \(n\geq2\), assuming only \(L^2\) regularity in the
normal direction and \(H^s\) regularity in the tangential directions,
where \(s>(n-1)/2\).
Finally, Section~6 extends the global \(H^1\) theory for
sphere-supported nonlinearities to all dimensions \(n\geq3\)
with radial initial data.

\subsection*{Acknowledgments} 

The authors are grateful to Benjamin Harrop-Griffiths for suggesting the problem, as well as his guidance and support. Part of this work was completed while the authors participated in the workshop \emph{Dispersive integrable equations: Pathfinders in Hamiltonian PDE} at the Institut Henri Poincar\'{e}. The authors are grateful to the organizers of this workshop for their invitation and the IHP for their hospitality. This material is based upon work supported by the U.S. National Science Foundation under award No. DMS-2406816. The authors acknowledge the use
of AI tools. All mathematical arguments and proofs in the final manuscript were checked and
written by the authors.

\section{Preliminaries}\label{sec:prelim}
We start by discussing our notation and some elementary inequalities. Our convention for the inner product on $L^2(\R^n)$ will be
\[\langle f,g \rangle := \int_{\R^n} \overline{f(x)}g(x)\,dx.\] 
For a Banach space $X$ and an interval $I$, we write $C(I;X) $ for the space of bounded continuous functions $\psi:I \to X$ equipped with the the supremum norm.
\[
\|\psi\|_{C(I;X)}:=\sup_{t\in I}\|\psi(t)\|_X,
\]
and write \(C(\R) = C(\R;\R)\).
We write $\M$ to  denote the space of finite signed Borel measures on $\R$, equipped with the total variation norm. We note that $\M$ is the dual space of the closed subspace $C_0(\R) \subseteq C(\R)$ of continuous functions vanishing at infinity.

For \(x\in \R^n\), we denote \(x=(x',y)\) where \(x'=(x_1,x_2,\cdots,x_{n-1})\in\R^{n-1}\). When using mixed spacetime norms, we will use subscripts to indicate the corresponding variable, for example $C_tL^p_{x'}L^q_{y}  = C(\R;L^p(\R_{x'}^{n-1};L^q(\R_{y})))$, and if \(I\) is a some bounded time interval then $C_IL^p_yL^q_{x'}  = C(I;L^p(\R_{y};L^q(\R_{x'}^{n-1})))$. For \(T>0\), we use \(C_T\) to denote $C_I$ with $I = [-T,T]$.

We adopt the unitary normalization of the Fourier transform of a function \(f\colon \R^n\to \R\), i.e.,
\[
\hat f(\xi) = (2\pi)^{-\frac n2}\int_{\R^n}f(x)e^{-ix\cdot \xi}\,dx.
\]

We recall that \(H^s(\R^{n})\) is an algebra whenever \(s>\frac{n}2\), i.e.,
\begin{equation}\label{eq:algebraic}
        \|fg\|_{H^s}\lesssim \|f\|_{H^s}\|g\|_{H^s}.
    \end{equation}

For \(M=\R\) or \(\R^+\), the space \(H^s(M)\) (and \(H^s_0(M)\) when \(M=\R^+\)) is an algebra as long as \(s>\frac 1 2\). By duality, we have
\begin{equation}\label{eq:dualalgebraic1d}
        \|fg\|_{H^{-s}(M)}\lesssim \|f\|_{H^{-s}(M)}\|g\|_{H^s(M)},
    \end{equation}
which will be repeatedly used in our proofs.

Let \(M=\R\), \(\R^+\), or \(\T = \R/(2\pi\Z)\). The 1-d Gagliardo--Nirenberg inequality ensures that for \(s>\frac12\) we have the embedding \(H^s(M)\hookrightarrow C(M)\) with the estimate
\begin{equation}\label{1-d GNinfty}
    \|f\|_{L^\infty(M)}\lesssim \|f\|_{L^2(M)}^{1-\frac {1}{2s}}\|f\|_{H^s(M)}^{\frac 1 {2s}}\lesssim \|f\|_{H^s(M)},
\end{equation}
and for \(s\geq \frac14\) we have the embedding \(H^s(M)\hookrightarrow L^4(M)\), with the estimate
\begin{equation}\label{1-d GN4}
    \|f\|_{L^4(M)}\lesssim \|f\|_{L^2(M)}^{1-\frac {1}{4s}}\|f\|_{H^s(M)}^{\frac 1 {4s}}\lesssim \|f\|_{H^s(M)}.
\end{equation}
By duality, \eqref{1-d GNinfty} and \eqref{1-d GN4} give us
\begin{equation}\label{1-d GNinfty dual}
    \|f\|_{H^{-s}(M)}\lesssim \|f\|_{L^1(M)}\quad \text{for \(s>\tfrac 1 2\)},
\end{equation}
and
\begin{equation}\label{1-d GN4 dual}
    \|f\|_{H^{-s}(M)}\lesssim \|f\|_{L^\frac{4}{3}(M)}\quad\text{for \(s\geq\tfrac 1 4\)}.
\end{equation}
Finally, we recall the dispersive estimate for the linear Schr\"odinger propagator in \(n\) dimensions is
\begin{equation}\label{dispersive}
    \|e^{it\Delta} f\|_{L^\infty(\R^n)}\lesssim |t|^{-\frac{n}{2}}\|f\|_{L^1(\R^n)};
\end{equation}
see, e.g., \cite{Tao2006NonlinearDispersive},\cite{cazenave2003semilinear}.

We will require the following slight refinement of \eqref{dispersive}, where \(f\) is taken to be a finite measure:

\begin{lem}\label{dispersive-M}For all \(n\geq 1\) and \(t\neq 0\) we have the dispersive estimate
    \[
    \|e^{it\Delta}f\|_{{L^\infty(\R^n)}}\lesssim |t|^{-\frac n2}\|f\|_{\M(\R^n)}.
    \]
\end{lem}
\begin{proof}
    By duality, it suffices to show that for \(f\in L^1(\R^n)\) we have
    \eq{dispersive for M}{
    \|e^{it\Delta} f\|_{C_0(\R^n)}\lesssim |t|^{-\frac{n}{2}}\|f\|_{L^1(\R^n)}.
    }

    We recall that for \(t\neq 0\) we have
    \begin{align*}
        (e^{it\Delta}f)(x)&=\frac{1}{(4\pi it)^{\frac{n}{2}}}\int_{\R^n}e^{i\frac{|x-y|^2}{4t}}f(y)dy
        =\frac{1}{(4\pi it)^{\frac{n}{2}}}e^{i\frac{|x|^2}{4t}}\int_{\R^n}e^{-i\frac{x}{2t}\cdot y}e^{i\frac{|y|^2}{4t}}f(y)dy
        =\frac{1}{(2 it)^{\frac{n}{2}}}e^{i\frac{|x|^2}{4t}}\hat {g_t}\left(\frac{x}{2t}\right),
    \end{align*}
    where \(g_t(y)=e^{i\frac{|y|^2}{4t}}f(y)\).
    
    It is clear that \(g_t\in L^1\). By the Riemann-Lebesgue lemma we have \(\hat {g_t}\in C_0\) and therefore \(e^{it\Delta}f\in C_0\) with the estimate \eqref{dispersive for M}.
\end{proof}

In the proof of Theorems~\ref{thrm:GWPy} and~\ref{thrm:radial} we will construct solutions to the singular problem by passing to the limit in a family of regularized equations. Let \(\rho\in C_c^\infty(\R^n)\) be non-negative, radial, supported in the unit ball \(|x|<1\), with
\(\int_{\R^n}\rho(x)\,dx=1\), and for \(\epsilon>0\) define
\begin{equation}\label{mollification}
\rho^\epsilon_\Sigma(x):=\int_\Sigma \epsilon^{-n}\rho\bigl(\tfrac {x-y}\epsilon\bigr)\,dS(y).
\end{equation}
We then consider  solutions \(\psi^\epsilon: I_\epsilon \times \R^n \to \C\) of the regularized equation, 
\begin{equation}\label{nls-epsilon}\tag{NLS$_{\epsilon}$}
    i\p_t\psi^\epsilon+\Delta\psi^\epsilon=2\mu\rho^\epsilon_\Sigma|\psi^\epsilon|^2\psi^\epsilon.
\end{equation}
When the choice of \(\Sigma\) is understood, we write \(\rho^\epsilon = \rho_\Sigma^\epsilon\).

This regularized equation is a standard semilinear Schrödinger
equation with \(H^1\)-subcritical nonlinearity for \(1\leq n\leq 3\). By the standard \(H^1\)-theory for semilinear Schrödinger
equations (see Cazenave~\cite{cazenave2003semilinear}) the regularized problem admits a
unique \(H^1\)-solution \(\psi^\epsilon\in C(I_\epsilon;H^1(\R^n))\) on its maximal lifespan \(I_\epsilon \subseteq \R\), such that the mass
\[
M(\psi^\epsilon)
:=
\frac12\|\psi^\epsilon\|_{L^2(\R^n)}^2,
\]
and energy
\[
E_\epsilon(\psi^\epsilon)
:=
\frac12\|\nabla\psi^\epsilon\|_{L^2(\R^n)}^2
+
\frac{\mu}{2}\int_{\R^n}
\rho^\epsilon_\Sigma(x)|\psi^\epsilon(x)|^4\,dx,
\]
are conserved. When \(n\geq 4\), we may use radiality to prove an analogous well-posedness result; see Lemma \ref{product-highd-radial}.

Our next lemma proves that under suitable hypotheses the maximal lifespan of our regularized solution is \(I_\epsilon = \R\) and we have uniform (in \(\epsilon\)) estimates for the \(H^1\)-norm.

\begin{lem}
\label{lem:uniform-regularized-control}
Let \(n\geq1\), \(\mu=\pm1\), for some \(\epsilon_0>0\) ,and suppose that
\begin{equation}\label{9-15-20-48}
 \sup_{0<\epsilon\leq \epsilon_0}\int_{\R^n}
\rho^\epsilon_\Sigma(x)|\psi(x)|^4\,dx
\lesssim
\|\psi\|_{H^1(\R^n)}^4
\end{equation}
for all
\(\psi\in H^1(\R^n)\) when \(1\leq n\leq3\) and all radial
\(\psi\in H^1(\R^n)\) when \(n\geq4\). 

Let \(\psi^\epsilon(0)\in H^1(\R^n)\) and \(\psi^\epsilon\in C(I_\epsilon;H^1(\R^n))\) be the corresponding maximal lifespan solution of \eqref{nls-epsilon}, where we assume that \(\|\psi^\epsilon(0)\|_{H^1}\) is sufficiently small in the focusing (\(\mu=-1\)) case and \(\psi^\epsilon(0)\) is radial when \(n\geq 4\). Then \(I_\epsilon = \R\) and we have the estimates
\begin{alignat}{3}
\sup_{t\in \R} \|\psi^\epsilon(t)\|_{H^1(\R^n)}^2 &\lesssim \|\psi^\epsilon(0)\|_{H^1}^2+\|\psi^\epsilon(0)\|_{H^1}^4\qquad&&\text{if \(\mu=+1\)},\label{9-15-20-52}\\
\sup_{t\in \R}\|\psi^\epsilon(t)\|_{H^1(\R^n)}^2 &\leq 2\|\psi^\epsilon(0)\|_{H^1}^2&&\text{if \(\mu=-1\)},\label{9-15-21-11}
\end{alignat}
uniformly for all \(0<\epsilon\leq 1\).

\end{lem}

\begin{proof}
We denote
\eq{V-epsilon definition}{
V_\epsilon(\psi)
:=
\int_{\R^n}
\rho^\epsilon_\Sigma(x)|\psi(x)|^4\,dx.
}
From \eqref{9-15-20-48}, we may find some \(C>0\), independent of \(\epsilon\), such that \(V_\epsilon(\psi)\leq C\|\psi\|_{H^1}^4\). Conservation of mass and energy gives us
\begin{equation}
\label{combined}
\|\psi^\epsilon(t)\|_{H^1(\R^n)}^2
+
\mu V_\epsilon(\psi^\epsilon(t))
=
\|\psi^\epsilon(0)\|_{H^1(\R^n)}^2
+
\mu V_\epsilon(\psi^\epsilon(0)),
\end{equation}
for all \(t\in I_\epsilon\).

Suppose first that \(\mu=1\). Since
\(V_\epsilon(\psi^\epsilon(t))\geq0\), we obtain
\[
\|\psi^\epsilon(t)\|_{H^1(\R^n)}^2
\leq
\|\psi^\epsilon(0)\|_{H^1(\R^n)}^2
+
V_\epsilon(\psi^\epsilon(0))
\leq
\|\psi^\epsilon(0)\|_{H^1}^2+C\|\psi^\epsilon(0)\|_{H^1}^4.
\]
Taking the supremum over \(t\in I_\epsilon\)
proves that
\[
\sup_{t\in I_\epsilon}\|\psi^\epsilon(t)\|_{H^1(\R^n)}^2\lesssim \|\psi^\epsilon(0)\|_{H^1}^2 + \|\psi^\epsilon(0)\|_{H^1}^4,
\]
uniformly for all \(0<\epsilon\leq \epsilon_0\). It follows that \(I_\epsilon = \R\) for all \(0<\epsilon\leq \epsilon_0\), with the estimate \eqref{9-15-20-52}.

Now suppose that \(\mu=-1\). From
\eqref{combined} and
\(V_\epsilon(\psi^\epsilon(0))\geq0\), we obtain
\eq{9-15-21-00}{
\|\psi^\epsilon(t)\|_{H^1(\R^n)}^2
=\|\psi^\epsilon(0)\|_{H^1(\R^n)}^2-V_\epsilon(\psi^\epsilon(0))+V_\epsilon(\psi^\epsilon(t))\leq\|\psi^\epsilon(0)\|_{H^1}^2+C\|\psi^\epsilon(t)\|_{H^1(\R^n)}^4.
}

Choose \(\delta>0\) such that
\(
16C\delta^2\leq1,
\)
and suppose that
\eq{9-15-21-03}{
\|\psi^\epsilon(0)\|_{H^1}^2\leq \delta^2.
}

For each \(0<\epsilon\leq \epsilon_0\) we have
\(\psi^\epsilon\in C(I_\epsilon;H^1(\R^n))\)
and hence the bootstrap bound
\[
\|\psi^\epsilon(t)\|_{H^1(\R^n)}^2\leq 4\|\psi^\epsilon(0)\|_{H^1(\R^n)}^2,
\]
holds on a neighborhood of \(0\) in \(I_\epsilon\).
Whenever this bound holds, the estimate \eqref{9-15-21-00}, our choice of \(\delta\), and the smallness of the inital data \eqref{9-15-21-03} gives us
\[
\|\psi^\epsilon(t)\|_{H^1(\R^n)}^2\leq \|\psi^\epsilon(0)\|_{H^1(\R^n)}^2+C\|\psi^\epsilon(t)\|_{H^1(\R^n)}^4\leq \|\psi^\epsilon(0)\|_{H^1(\R^n)}^2+16C\delta^2\|\psi^\epsilon(0)\|_{H^1}^2\leq
2\|\psi^\epsilon(0)\|_{H^1}^2.
\]
Thus, the bootstrap bound improves to
\[
\|\psi^\epsilon(t)\|_{H^1(\R^n)}^2\leq 2\|\psi^\epsilon(0)\|_{H^1}^2<4\|\psi^\epsilon(0)\|_{H^1}^2.
\]
A standard continuity argument now closes the bootstrap on
\(I_\epsilon\), yielding
\[
\sup_{t\in I_\epsilon}\|\psi^\epsilon(t)\|_{H^1(\R^n)}^2\leq2\|\psi^\epsilon(0)\|_{H^1}^2.
\]
It follows that \(I_\epsilon=\R\). Since the constants are independent of \(\epsilon\), the estimate \eqref{9-15-21-11} follows.
\end{proof}

We will will also require the following:

\begin{lem}
    
\label{lem:Aepsilon}
Let $\Sigma$ be either $\R\times\{0\}$ or $\bbS^1\subset\R^2$.
Let $\rho^\epsilon_\Sigma$ be defined as in \eqref{mollification}.
Suppose $Q\subset H^1(\R^2)$ is bounded. Then,
$$
\lim_{\epsilon\to0}
\sup_{\psi\in Q}
\left|
\int_{\R^2}\rho_\Sigma^\epsilon(x)|\psi(x)|^4\,dx
-
\int_\Sigma |\psi|^4\,dS
\right|
=0.
$$
\end{lem}

\begin{proof}
Using the definition \eqref{mollification} with the change of variable $x = y + \epsilon z$, we can rewrite 
\[\int_{\R^2}\rho_\Sigma^\epsilon(x)|\psi(x)|^4\,dx=\int_\Sigma\int_{\R^2}
\rho(z)|\psi(y+\epsilon z)|^4\,dz\,dS(y).\]
Now, since $\int_{\R^2}\rho(z)\,dz=1$ we can express the difference as
\[
\int_{\R^2}\rho_\Sigma^\epsilon(x)|\psi(x)|^4\,dx - \int_\Sigma |\psi|^4\,dS
=
\int_{\R^2}\rho(z) \int_\Sigma \left(|\tau_{\epsilon z}\psi|^4 - |\psi|^4\right) dS\,dz,
\]
where $\tau_h\psi(y) := \psi(y+h)$. 
By the trace theorem on manifold \cite{McLean2000} and Sobolev embedding, $\|u\|_{L^4(\Sigma)} \lesssim \|u\|_{H^{3/4}(\R^2)}$. Since $Q$ is bounded in $H^1(\R^2)$, the $H^{3/4}$-norms of $\psi$ and its translations are uniformly bounded. Then $$\|\tau_h\psi - \psi\|_{H^{3/4}(\R^2)} \lesssim |h|^{1/4} \|\psi\|_{H^1(\R^2)},$$ and hence
\[\left| \int_\Sigma \left(|\tau_{\epsilon z}\psi|^4 - |\psi|^4\right) dS \right|\lesssim_Q
\|\tau_{\epsilon z}\psi - \psi\|_{H^{3/4}(\R^2)}\lesssim_Q
\epsilon^{1/4}|z|^{1/4}.
\]
Integrating against $\rho(z)$, and letting $\epsilon \to 0$ yields the result.
\end{proof}

\section{Global Wellposedness for \eqref{nls-y} in \(H^1(\R^2)\)}
In this section we prove Theorem \ref{thrm:GWPy} for case of the plane, \eqref{nls-y}.

We first define the Littlewood--Paley projections that will be used frequently in this section.
Given dyadic \(N\geq 1\), the partial (in \(y\)) and full Littlewood--Paley projections are defined as
    \begin{align*}
        \widehat{P^y_{\leq N}f}(x,\eta)&:=\varphi\left(\frac{\eta}{N}\right)\widehat{f}(x,\eta)\\
        \widehat{P_{\leq N}f},(\xi,\eta)&:=\varphi\left(\frac{\sqrt{\xi^2+\eta^2}}{N}\right)\widehat{f}(\xi,\eta),
    \end{align*}
    where \(\varphi\in C_c^\infty(-2,2)\) is a non-negative bump function that is identically \(1\) on \([-1,1]\). We denote \(P^y_{>N}:=1-P^y_{\leq N}\) and \(P_{>N}:=1-P_{\leq N}\).
    
We recall the following Bernstein inequalities, which are proved in \cite{Tao2006NonlinearDispersive}:
\begin{lem}\label{Bernstein_y}For all dyadic \(N\geq 1\) we have the estiamtes
\begin{alignat}{3}
    &\|P^y_{\leq N}f\|_{L^\infty_y}\lesssim N^{-s+\frac 1 2}\|P^y_{\leq N}f\|_{H^s_y}\lesssim N^{-s+\frac 1 2}\|f\|_{H^s_y} \qquad &&\text{if }s<\frac 1 2,\label{bernie down}\\
    &\|P^y_{> N}f\|_{L^\infty_y}\lesssim N^{-s+\frac 1 2}\|P^y_{> N}f\|_{H^s_y}\lesssim N^{-s+\frac 1 2}\|f\|_{H^s_y} \qquad &&\text{if }s>\frac 1 2,\label{bernie up}\\
    &\|P_{\leq N}f\|_{H^k(\R^2)}\lesssim N^{-s+k}\|P_{\leq N}f\|_{H^s(\R^2)}\lesssim N^{-s+k}\|f\|_{H^s(\R^2)} \qquad &&\text{if }s<k.
\end{alignat}
\end{lem}

Next, we prove an anisotropic dispersive estimate:
\begin{lem}\label{lem:dispersive_y} For all \(t\neq 0\) and \(f\in L^1_yL^{4/3}_x\), we have the estimate
\begin{align}\label{dispersive estimate for nls-y(C)}
    \|e^{it\Delta}f\|_{C_{0,y}L^4_x}\lesssim |t|^{-\frac 3 4}\|f\|_{L^1_yL^{4/3}_x},
    \end{align}
    and by duality, for all \(f\in\M_yL^{4/3}_x\), we have the estimate
    \begin{align}\label{dispersive estimate for nls-y}
    \|e^{it\Delta}f\|_{L^\infty_yL^4_x}\lesssim |t|^{-\frac 3 4}\|f\|_{\M_yL^{4/3}_x}.
    \end{align}
\end{lem}
\begin{proof}
To prove \eqref{dispersive estimate for nls-y(C)}, we assume by density that \(f\in \Schwartz\). From Lemma~\ref{dispersive-M}, for \(t\neq0\) we have \(e^{it\Delta}f\in C_0(\R^2)\). Consequently, it suffices to show that 
\eq{mixed norm dispersive}{
\|e^{it\Delta}f\|_{L^\infty_{y}L^4_x}\lesssim |t|^{-\frac 3 4}\|f\|_{L^1_yL^{4/3}_x}.
}

    Using the fact that \(e^{it\p_x^2}\) is an isometry in \(L^2_x\) and applying \eqref{dispersive}, we have
    \begin{align*}
\|e^{it\Delta}f\|_{L^\infty_yL^2_x}=\|e^{it\p_y^2}f\|_{L^\infty_yL^2_x}\lesssim\|e^{it\p_y^2}f\|_{L^2_xL^\infty_y}\lesssim|t|^{-\frac 1 2}\|f\|_{L^2_xL^1_y}\lesssim |t|^{-\frac 1 2}\|f\|_{L^1_yL^2_x}.
    \end{align*}
    Moreover, applying  \eqref{dispersive} for the 2-d case, we obtain
    \[
    \|e^{it\Delta}f\|_{L^{\infty}(\R^2)}\lesssim |t|^{-1}\|f\|_{L^1(\R^2)}.
    \]

    The estimate \eqref{mixed norm dispersive} now follws from applying the Riesz--Thorin interpolation theorem for mixed norms; see \cite[Theorem 2]{Blozinski1981MultivariateRA}.

\end{proof}

By applying Sobolev embedding \eqref{1-d GNinfty} and the Gagliardo--Nirenberg inequality \eqref{1-d GN4}, we have the anisotropic Sobolev embedding \(H^1(\R^2) \hookrightarrow C_{0,y}L_x^4\), which will be used several times in this section. To prove this, we observe that for any \(\epsilon\in (0,\frac 1 4]\) we have
\begin{equation}\label{eq:anisotrophic-y}
    \|f\|_{L^\infty_yL^4_x}\lesssim \|\langle \p_x\rangle^{1/4}f\|_{L^\infty_yL^2_x} \lesssim \| \langle \p_y \rangle^{\frac 1 2+\epsilon}\langle \p_x\rangle^{\frac 14}f\|_{L^2(\R^2)}\lesssim\|f\|_{H^{\frac 3 4+\epsilon}}\lesssim \|f\|_{H^1(\R^2)}.
\end{equation}
We also have the dual estimate
\begin{equation}\label{eq:anisotrophic-y-dual}
    \|f\|_{H^{-1}(\R^2)}\lesssim \|f\|_{L^1_yL^\frac{4}{3}_x}.
\end{equation}

We now use \eqref{dispersive estimate for nls-y} to prove a Lipschitz estimate for the solution map of \eqref{nls-y}, which yields uniqueness as an immediate corollary:
\begin{prop}\label{K-estimate}
    Suppose that for some \(T>0\) there are two solutions \(\psi,\phi\in C([-T,T];H_w^1(\R^2)\cap L_y^\infty L_x^4)\) of the Duhamel formulation \eqref{eq:Dulm} of \eqref{nls-y}, where \(H_w^1(\R^2)\) denotes \(H^1(\R^2)\) endowed with the weak topology. Then for all \(|t|\leq |T|\) we have the estimate
    \begin{equation}\label{eq:p3.1}
        \|\psi(t)-\phi(t)\|_{H^{-1}(\R^2)}\leq (1+K|t|e^{CK^8t^2})\|\psi(0)-\phi(0)\|_{H^1(\R^2)},
    \end{equation}
    where \(C>0\) is some absolute constant and
    \[
    K:=\sup_{|t|\leq T}\Bigl\{\|\psi(t)\|^2_{L^{\infty}_yL^4_x}+\|\phi(t)\|^2_{L^{\infty}_yL^4_x}\Bigr\}\lesssim \sup_{|t|\leq T}\Bigl\{\|\psi(t)\|^2_{H^1(\R^2)}+\|\phi(t)\|^2_{H^1(\R^2)}\Bigr\}.
    \]
\end{prop}
\begin{proof}
    By the Duhamel formula, for \(0<t\leq T\) we have
    \[
    \psi(t)-\phi(t)=e^{it\Delta}(\psi(0)-\phi(0))-i2\mu \int_0^t e^{i(t-s)\Delta}\nonlineary\left[|\psi(s)|^2\psi(s)-|\phi(s)|^2\phi(s)\right]\,ds.
    \]
    For the first term, we use the anisotropic Sobolev embedding \eqref{eq:anisotrophic-y} to bound
    \[
    \|e^{it\Delta}(\psi(0)-\psi(0))\|_{L^\infty_yL^4_x}\lesssim \|e^{it\Delta}(\psi(0)-\phi(0))\|_{H^1(\R^2)}\lesssim\|\psi(0)-\phi(0)\|_{H^1(\R^2)}.
    \]
    For the second term, we apply \eqref{dispersive estimate for nls-y} to estimate
    \begin{align*}
       &\left\|\int_0^t e^{i(t-s)\Delta}\nonlineary\left[|\psi(s)|^2\psi(s)-|\phi(s)|^2\phi(s)\right]ds\right\|_{L^\infty_yL^4_{x}} \\
       &\ \ \ \ \ \ \lesssim \int_0^t |t-s|^{-\frac 3 4}\|\nonlineary\left[|\psi(s)|^2\psi(s)-|\phi(s)|^2\phi(s)\right]\|_{\M_yL_x^{4/3}}ds\\
       &\ \ \ \ \ \ \lesssim \int_0^t |t-s|^{-\frac 3 4}(\|\psi(s)\|^2_{L_y^\infty L_x^4}+\|\phi(s)\|^2_{L_y^\infty L_x^4})\|\psi(s)-\phi(s)\|_{L_y^\infty L_x^4}ds\\
       &\ \ \ \ \ \ \lesssim K\int_0^t |t-s|^{-\frac 3 4}\|\psi(s)-\phi(s)\|_{L_y^\infty L_x^4}ds \\
       &\ \ \ \ \ \ \lesssim K\left(\int_0^t|t-s|^{-\frac 6 7}ds\right)^{\frac 7 8}\left(\int_0^t\|\psi(s)-\phi(s)\|_{L_y^\infty L_x^4}^8ds\right)^{\frac 1 8}\\
       &\ \ \ \ \ \ \lesssim \left(K^{8}t\int_0^t\|\psi(s)-\phi(s)\|_{L_y^\infty L_x^4}^8ds\right)^{\frac 1 8}.
    \end{align*}
    Therefore, by taking the 8\textsuperscript{th} power, we have
    \[
   \|\psi(t)-\phi(t)\|_{L_y^\infty L_x^4}^8\lesssim \|\psi(0)-\phi(0)\|_{H^1(\R^2)}^8+K^{8}t\int_0^t\|\psi(s)-\phi(s)\|_{L_y^\infty L_x^4}^8\,ds.
    \]
    The integral form of Gronwall's inequality then give us
    \[
   \|\psi(t)-\phi(t)\|_{L_y^\infty L_x^4}\lesssim e^{CK^{8}t^2}\|\psi(0)-\phi(0)\|_{H^1(\R^2)}.
    \]
    Moreover, for any test function \(f\in H^1(\R^2)\),
    \begin{align*}
        \bigl|\langle\nonlineary\left(|\psi|^2\psi-|\phi|^2\phi\right),f\rangle\bigr|&=\left|\int_\R\overline{\left(|\psi(x,0)|^2\psi(x,0)-|\phi(x,0)|^2\phi(x,0)\right)}f(x,0)dx\right|\\
        &\lesssim (\|\psi(x,0)\|_{L^4_x}^2+\|\phi(x,0)\|_{L^4_x}^2)\|\psi(x,0)-\phi(x,0)\|_{L_x^4}\|f(x,0)\|_{L^4_x}\\
        &\lesssim K\|\psi-\phi\|_{L^\infty_yL^4_x}\|f\|_{L^\infty_yL^4_x}\\
        &\lesssim K\|\psi-\phi\|_{L^\infty_yL^4_x}\|f\|_{H^1(\R^2)}.
    \end{align*}
    By duality, we conclude that
    \[
    \bigl\|\nonlineary\left(|\psi|^2\psi-|\phi|^2\phi\right)\bigr\|_{H^{-1}(\R^2)}\lesssim K \|\psi-\phi\|_{L^\infty_yL^4_x}.
    \]
    Finally, we apply this estimate to obtain
    \begin{align*}
        \|\psi(t)-\phi(t)\|_{H^{-1}(\R^2)}&\lesssim \|\psi(0)-\phi(0)\|_{H^{-1}(\R^2)}+\int_0^t \bigl\|\nonlineary\left[|\psi(s)|^2\psi(s)-|\phi(s)|^2\phi(s)\right]\bigr\|_{H^{-1}(\R^2)}ds\\
        &\lesssim\|\psi(0)-\phi(0)\|_{H^{1}}+K\int_0^t \|\psi(s)-\phi(s)\|_{L_y^\infty L_x^4}ds\\
        &\lesssim (1+Kte^{CK^8t^2})\|\psi(0)-\phi(0)\|_{H^1(\R^2)},
    \end{align*}
    which yields \eqref{eq:p3.1} for \(0<t\leq T\). The corresponding estimate for times \(-T\leq t<0\) follows from a parallel argument.
\end{proof}

Uniqueness of solutions to \eqref{nls-y} also yields the following:

\begin{lem}\label{Conservation-limiting}
Let \(\psi\in C([0,T];H_w^1(\R^2)\cap L_y^\infty L_x^4)\) be a solution of \eqref{nls-y} where $\mu=\pm1$ satisfying the inequalities
\[
M(\psi(t))\leq M(\psi(0)),
\qquad
E(\psi(t))\leq E(\psi(0)),
\]
for all \(0\leq t\leq T\).

Suppose that for each $t_0\in[0,T]$ there exists a solution \(\phi\in C([0,t_0];H_w^1(\R^2)\cap L_y^\infty L_x^4)\) of \eqref{nls-y} satisfying \(\phi(0) = \overline{\psi(t_0)}\) as well as the inequalities
\[
M(\phi(t))\leq M(\phi(0)),
\qquad
E(\phi(t))\leq E(\phi(0)),
\]
for all \(0\leq t\leq t_0\).

Then \(\psi\) conserves both mass and energy on \([0,T]\), i.e,
\[
M(\psi(t))=M(\psi(0)),
\qquad
E(\psi(t))=E(\psi(0)),
\]
for all \(0\leq t\leq T\).
\end{lem}

\begin{proof}

By assumption, we have
\[
E(\psi(t))\leq E(\psi(0))\qquad \text{for }0\leq t\leq T.
\]
Fix $t_0\in[0,T]$ and consider
\[\widetilde\psi(s,x)=\overline{\psi(t_0-s,x)}.\]
By time-reversal symmetry, $\widetilde\psi$ is also a
solution of \eqref{nls-y} on $[0,t_0]$ and has initial data \(\widetilde\psi(0) = \overline{\psi(t_0)} = \phi(0)\).

Observing that Proposition~\ref{K-estimate} holds with \([-T,T]\) replaced by \([0,t_0]\), we conclude that \(\widetilde\psi = \phi\) and hence
\[
E(\widetilde\psi(t_0))\leq E(\widetilde\psi(0)).\]
However,
\[\widetilde\psi(0)=\overline{\psi(t_0)}\qtq{and}\widetilde\psi(t_0)=\overline{\psi(0)}.
\]
Since the energy is invariant under complex conjugation, i.e.,
\(E(\overline \psi)=E(\psi)\),
we have
\[
E(\psi(0))=E(\overline{\psi(0)})=E(\widetilde\psi(t_0))\leq E(\widetilde\psi(0))=E(\overline{\psi(t_0)})=E(\psi(t_0))\leq E(\psi(0)),\]
and hence
\[E(\psi(t_0))=E(\psi(0)).\]
Since $t_0\in[0,T]$ was arbitrary, energy is conserved on $[0,T]$.

The proof for the mass follows is similar.
\end{proof}

Next, we turn to the problem of existence. 

We recall that we will construct solutions to \eqref{nls-y} from the \(\epsilon \to0\) limit of solutions of the regularized problem \eqref{nls-epsilon}. 
existence of local-in-time solutions to \eqref{nls-y}.

To do this we require one additional preliminary lemma:
\begin{lem}\label{l:rho-diff}
For \(s>\frac 1 2\) and \(0<\epsilon<\epsilon'<1\) we have the estimate,
    \eq{rho-diff}{
        \|\rho^\epsilon -\rho^{\epsilon'} \|_{H^{-s}_y}\lesssim |\epsilon-\epsilon'|^{s-\frac 1 2}
        }
\end{lem}
\begin{proof}
Write \(\rho^\epsilon(x,y) = \varrho^\epsilon(y) = \frac1\epsilon \varrho(\frac y\epsilon)\).
    For any test function \(f\in C_c^\infty(\R)\), we may then apply Plancherel's Theorem to bound 
        \begin{align*}
            \bigl|\langle f,\rho^\epsilon -\rho^{\epsilon'} \rangle\bigr|&=\left|\int_\R (\varrho^\epsilon(y) -\varrho^{\epsilon'}(y) )\overline{f(y)}\,dy\right|\\
            &=\left|\int_\R(\hat\varrho(\epsilon\xi)-\hat\varrho(\epsilon'\xi))\overline{\hat{f}(\xi)}\,d\xi\right|\\
            &=\left|\int_\R\left(\int_{\epsilon'}^\epsilon\xi\hat\varrho'(\xi h)\,dh\right)\overline{\hat{f}(\xi)}\,d\xi\right| \\
            &\lesssim \int_{\epsilon}^{\epsilon'}\|\xi^{1-s}\hat\varrho'(\xi h)\|_{L^2_\xi}\|f(y)\|_{\dot H^{s}_y}\,dh\\
            &\lesssim \left(\int_{\epsilon}^{\epsilon'} h^{s-\frac 3 2}\,dh\right) \|f(y)\|_{\dot H^{s}_y}\\
            &\lesssim|\epsilon-\epsilon'|^{s-\frac 1 2}\|f(y)\|_{ H^{s}_y},
        \end{align*}
        where the implicit constant depends on \(\rho\). The estimate \eqref{rho-diff} now follows by duality.
\end{proof}

\begin{prop}\label{p:10-15}
    Given \(\psi_0\in H^1(\R^2)\), there exists \(T=T(\|\psi_0\|_{H^1(\R^2)})>0\) and a solution \(\psi\in C([-T,T];H^1(\R^2))\) of the Duhamel formulation \eqref{eq:Dulm} of \eqref{nls-y} in the defocusing case, or in the focusing case when \(\|\psi_0\|_{H^1(\R^2)}\) is sufficiently small. Moreover, if \(\psi^\epsilon\in C(\R;H^1(\R^2))\) is a family of solutions to \eqref{nls-epsilon} with \(\psi^\epsilon(0)=\psi_0\in H^1(\R^2)\), then
    \[
    \lim_{\epsilon\to 0}\sup_{|t|\leq T}\|\psi^{\epsilon}(t)-\psi(t)\|_{H^1(\R^2)}=0
    \]
\end{prop}
\begin{proof}
    We will construct \(\psi(t)\) as a limit of solutions \(\psi^\epsilon(t)\) to \eqref{nls-epsilon}.

    Given \(T>0\), we define the control quantity
    \[
    A:=\sup_{0<\epsilon\leq1}\sup_{|t|\leq T}\|\psi^\epsilon(t)\|_{H^1(\R^2)}.
    \]
Recalling the definition of \(V_\epsilon\) in
Lemma~\ref{lem:uniform-regularized-control}, we have
\[
V_\epsilon(\psi^\epsilon)
=
\int_\Sigma\int_{\R^2}
\rho(z)|\psi^\epsilon(y+\epsilon z)|^4\,dz\,dS(y),
\]
where \(\Sigma = \{y=0\}\). By the trace theorem and Sobolev embedding, together with
translation invariance of the \(H^1(\R^2)\)-norm,
\eq{V4estimate}{
\|\tau_{\epsilon z}\psi^\epsilon\|_{L^4(\Sigma)}
\lesssim
\|\tau_{\epsilon z}\psi^\epsilon\|_{H^1(\R^2)}
=
\|\psi^\epsilon\|_{H^1(\R^2)}.
}
Since \(\rho\geq0\) and \(\int_{\R^2}\rho(z)\,dz=1\), it follows that
\[
0\leq V_\epsilon^\epsilon(\psi)
\lesssim
\|\psi^\epsilon\|_{H^1(\R^2)}^4,
\]
uniformly for all \(\epsilon>0\).
    Then, by Lemma~\ref{lem:uniform-regularized-control}, we know that \eqref{nls-epsilon} with initial data \(\psi^\epsilon(0) = \psi_0\) is globally well-posed in \(H^1(\R^2)\) and have the uniform estimate \(A < \infty\).
    
    Taking \(0<\epsilon<\epsilon'\leq 1\), we use the Duhamel formulation \eqref{eq:Dulm} to write
    \[
    \psi^{\epsilon}(t)-\psi^{\epsilon'}(t)=- i2\mu\int_0^t e^{i(t-s)\Delta}\left(F_1(s)+F_2(s)\right)ds,
    \]
    where we have used that \(\psi^\epsilon(0) = \psi_0 = \psi^{\epsilon'}(0)\) and denote
    \[
F_1:=\rho^{\epsilon}\left(|\psi^{\epsilon}|^2\psi^{\epsilon}-|\psi^{\epsilon'}|^2\psi^{\epsilon'}\right)
\qtq{and} F_2:=\left(\rho^{\epsilon}-\rho^{\epsilon'}\right)|\psi^{\epsilon'}|^2\psi^{\epsilon'}.
\]
Our first objective is to show that 
\begin{equation}\label{eq:L4Cy}
    \lim_{\epsilon,\epsilon'\to0}\sup_{|t|\leq T}\|\psi^{\epsilon}-\psi^{\epsilon'}\|_{L^\infty_yL^4_x}=0.
\end{equation}

    We decompose
\begin{align*}
            \sup_{|t|\leq T}\|\psi^{\epsilon}(t)-\psi^{\epsilon'}(t)\|_{L^\infty_yL^4_x}&\leq\sup_{|t|\leq T}\|P_{>N}^y(\psi^{\epsilon}(t)-\psi^{\epsilon'}(t))\|_{L^\infty_yL^4_x}+\sup_{|t|\leq T}\|P_{\leq N}^y(\psi^{\epsilon}(t)-\psi^{\epsilon'}(t))\|_{L^\infty_yL^4_x}\\
            &\lesssim\underbrace{\sup_{|t|\leq T}\|P_{>N}^y(\psi^{\epsilon}(t)-\psi^{\epsilon'}(t))\|_{L^\infty_yL^4_x}}_{=:I} +\underbrace{\sup_{|t|\leq T}\|P_{\leq N}^y\int_0^t e^{i(t-s)\Delta}F_1(s)ds\|_{L^\infty_yL^4_x}}_{=:II}\\
            & \ \ \ +\underbrace{\sup_{|t|\leq T}\|P_{\leq N}^y\int_0^t e^{i(t-s)\Delta}F_2(s)ds\|_{L^\infty_yL^4_x}}_{=:III}.
        \end{align*}

    For \(I\), we apply the first inequality in \eqref{eq:anisotrophic-y} and the Bernstein inequality \eqref{bernie up} to bound
        \begin{align*}
\|P_{>N}^yf\|_{L^\infty_yL^4_x}\lesssim \|\<\p_x\>^{\frac14}P_{>N}^yf\|_{L^\infty_yL^2_x}\lesssim \|\<\p_x\>^{\frac14}P_{>N}^yf\|_{L^2_xL^\infty_y}
\lesssim N^{-\frac 1 4}\|\langle\p_x\rangle^{\frac 1 4}\langle \p_y\rangle^{\frac 3 4}f\|_{L^2}\lesssim N^{-\frac 1 4}\|f\|_{H^1(\R^2)},
        \end{align*}
which yields
        \[
        I\lesssim N^{-\frac 1 4}A.
        \]

        For \(II\), we apply the dispersive estimate  \eqref{dispersive estimate for nls-y} to estimate
         \begin{align*}
             \sup_{|t|\leq T}\|P_{\leq N}^y\int_0^t e^{i(t-s)\Delta}F_1(s)ds\|_{L^\infty_yL^4_x}&\lesssim\sup_{|t|\leq T}\int_0^t |t-s|^{-\frac 3 4}\|F_1(s)\|_{L^1_yL_x^{4/3}}ds\\
             &\lesssim \sup_{|t|\leq T}\int_0^t |t-s|^{-\frac 3 4}\|\rho^\epsilon\|_{L_y^1}\|(\psi^\epsilon)^2+(\psi^{\epsilon'})^2\|_{L^\infty_yL_x^{2}}\|\psi^\epsilon-\psi^{\epsilon'}\|_{L^\infty_yL^4_x}ds\\
             \\&\lesssim T^{\frac 1 4}A^2\sup_{|t|\leq T}\|\psi^\epsilon(t)-\psi^{\epsilon'}(t)\|_{L^\infty_yL^4_x}.
         \end{align*}
        \eqref{dispersive estimate for nls-y}, \eqref{eq:dualalgebraic1d}, and \eqref{1-d GN4} to obtain
            \begin{align*}
            \sup_{|t|\leq T}\left\|\int_0^t e^{i(t-s)\Delta}F_2(s)ds\right\|_{L^4_xH^{-\frac 3 4}_{y}}&\lesssim\sup_{|t|\leq T}\int_0^t \left\|e^{i(t-s)\p_x^2}\<\p_y\>^{-\frac34}F_2(s)\right\|_{L_y^2L^4_x}\,ds\\
            &\lesssim\sup_{|t|\leq T}\int_0^t |t-s|^{-\frac 1 4}\left\|F_2(s)\right\|_{L^{\frac 4 3}_xH^{-\frac 3 4}_{y}}ds\\
            &\lesssim T^{\frac34}\sup_{|t|\leq T}\| F_2(s)\|_{L^{\frac 4 3}_xH^{-\frac 3 4}_{y}}\\
            &\lesssim T^{\frac 3 4}\sup_{|t|\leq T}\|\rho^\epsilon -\rho^{\epsilon'} \|_{H^{-\frac 3 4}_y}\||\psi^{\epsilon'}|^2\psi^{\epsilon'}\|_{L^{\frac 4 3}_xH^{\frac 3 4}_{y}}\\
            &\lesssim T^{\frac 3 4}\sup_{|t|\leq T}\|\psi^{\epsilon'}\|^3_{L^{4}_xH^{\frac 3 4}_{y}}|\epsilon-\epsilon'|^{\frac 1 4}\\
            &\lesssim T^{\frac 3 4}\sup_{|t|\leq T}\|\<\p_x\>^{\frac14}\<\p_y\>^{\frac34}\psi^{\epsilon'}\|^3_{L^2(\R^2)}|\epsilon-\epsilon'|^{\frac 1 4}\\
            &\lesssim T^{\frac 3 4}A^3|\epsilon-\epsilon'|^{\frac 1 4}.
        \end{align*}
        By Bernstein's inequality at low frequency \eqref{bernie down}, this yields
        \[
        \sup_{|t|\leq T}\left\|P_{\leq N}^y\int_0^t e^{i(t-s)\Delta}F_2(s)ds\right\|_{L^\infty_{y}L^4_x}\lesssim\sup_{|t|\leq T}\left\|P_{\leq N}^y\int_0^t e^{i(t-s)\Delta}F_2(s)ds\right\|_{L^4_xL^\infty_{y}}\lesssim N^{\frac 5 4}T^{\frac 3 4}A^3|\epsilon-\epsilon'|^{\frac 1 4}.
        \]
        
        Combining all of these estimates finally gives us
        \[
        \sup_{|t|\leq T}\|\psi^{\epsilon}(t)-\psi^{\epsilon'}(t)\|_{L^\infty_yL^4_x}\lesssim T^{\frac 1 4}A^2\sup_{|t|\leq T}\|\psi^\epsilon(t)-\psi^{\epsilon'}(t)\|_{L^\infty_yL^4_x}+N^{\frac 5 4}T^{\frac 3 4}A^3|\epsilon-\epsilon'|^{\frac 1 4}+N^{-\frac 1 4}A.
        \]
        
        We now choose \(T = T(\|\psi_0\|_{H^1(\R^2)})>0\) sufficiently small to ensure that
        \(
        TA^8\ll 1.
        \)
        We then have
        \[
        \limsup_{\epsilon,\epsilon'\to0}\sup_{|t|\leq T}\|\psi^{\epsilon}(t)-\psi^{\epsilon'}(t)\|_{L^\infty_yL^4_x}\lesssim N^{-\frac 1 4}A.
        \]
         Taking the limit as \(N\to\infty\), we see that \(\psi^\epsilon\) converges in \(C([-T,T];L^\infty_yL^4_x )\) and denote the limit by \(\psi\).
        
        Next, we will prove that
        \begin{equation}\label{eq:H-2}
                \lim_{\epsilon,\epsilon'\to0} \sup_{|t|\leq T}\|\psi^{\epsilon}(t)-\psi^{\epsilon'}(t)\|_{H^{-2}(\R^2)}=0.
        \end{equation}
        We again start by decomposing,
        \begin{align*}
            \sup_{|t|\leq T}\|\psi^{\epsilon}(t)-\psi^{\epsilon'}(t)\|_{H^{-2}(\R^2)}\lesssim\underbrace{\sup_{|t|\leq T}\|\int_0^t e^{i(t-s)\Delta}F_1(s)ds\|_{H^{-2}(\R^2)}}_{=:I}+\underbrace{\sup_{|t|\leq T}\|\int_0^t e^{i(t-s)\Delta}F_2(s)ds\|_{H^{-2}(\R^2)}}_{=:II}.
        \end{align*}
        
        For \(I\), we apply \eqref{eq:anisotrophic-y-dual} to bound
        \begin{align*}
            I
            &\lesssim T\sup_{|t|\leq T}\|F_1(s)\|_{L^1_yL^{4/3}_x}\\
            &\lesssim T\|\rho_{\epsilon}\|_{L^1_y}\left(\sup_{|t|\leq T}\|\psi^{\epsilon}(t)\|_{L^\infty_yL^4_x}^2+\sup_{|t|\leq T}\|\psi^{\epsilon'}(t)\|_{L^\infty_yL^4_x}^2\right)\sup_{|t|\leq T}\|\psi^{\epsilon}(t)-\psi^{\epsilon'}(t)\|_{L^\infty_yL^4_x}\\
            &\lesssim TA^2\sup_{|t|\leq T}\|\psi^{\epsilon}(t)-\psi^{\epsilon'}(t)\|_{L^\infty_yL^4_x}.
        \end{align*}

        By \eqref{rho-diff}, \eqref{eq:anisotrophic-y}, and \eqref{1-d GNinfty dual}, we have
        \begin{align*}
            II&\lesssim T\sup_{|t|\leq T}\|\<\p_x\>^{-1}\<\p_y\>^{-1}F_2(t)\|_{L^2(\R^2)}\\
            &\lesssim T\sup_{|t|\leq T}\{\|\rho^\epsilon -\rho^{\epsilon'} \|_{H^{-1}_y}\bigl\|\<\p_x\>^{-1}\<\p_y\>\bigl(|\psi^{\epsilon'}(t)|^2\psi^{\epsilon'}(t)\bigr)\bigr\|_{L^2(\R^2)}\}\\
            &\lesssim T|\epsilon-\epsilon'|^{\frac 1 2}\sup_{|t|\leq T}\{\|\|\psi^{\epsilon'}(t)\|_{L^\infty_y}^2\|\psi^{\epsilon'}(t)\|_{H^1_y}\|_{L^1_x}\}\\
            &\lesssim T|\epsilon-\epsilon'|^{\frac 1 2}\sup_{|t|\leq T}\{\|\psi^{\epsilon'}(t)\|_{L^4_xL^\infty_y}^2\|\psi^{\epsilon'}(t)\|_{L^2_xH^1_y}\}\\
            &\lesssim T|\epsilon-\epsilon'|^{\frac 1 2}A^3.
        \end{align*}
        
        Combining these estimates, we get
        \[
         \sup_{|t|\leq T}\|\psi^{\epsilon}(t)-\psi^{\epsilon'}(t)\|_{H^{-2}(\R^2)}\lesssim TA^2\sup_{|t|\leq T}\|\psi^{\epsilon}(t)-\psi^{\epsilon'}(t)\|_{L^\infty_yL^4_x}+T|\epsilon-\epsilon'|^{\frac 1 2}A^3.
        \]
        Passing the limit as \(\epsilon,\epsilon'\to 0\) shows that \(\psi^\epsilon\to\psi\) in \(C([-T,T];H^{-2}(\R^2))\)

        This allows us to show convergence in the space of continuous functions from \([-T,T]\) to \(H^1(\R^2)\) endowed with weak topology. For any test function \(\phi\in H^{-1}(\R^2)\), we have
        \[
        \sup_{|t|\leq T}|\langle\psi^\epsilon-\psi^{\epsilon'},\phi\rangle|\leq\sup_{|t|\leq T}|\langle\psi^\epsilon-\psi^{\epsilon'},P_{\leq N}\phi\rangle|+\sup_{|t|\leq T}|\langle\psi^\epsilon-\psi^{\epsilon'},P_{>N}\phi\rangle|.
        \]
        For the first term, we use Bernstein's inequality \eqref{bernie down} to get
        \begin{align*}
            \sup_{|t|\leq T}|\langle\psi^\epsilon-\psi^{\epsilon'},P_{\leq N}\phi\rangle|&\leq\sup_{|t|\leq T}\|\psi^\epsilon-\psi^{\epsilon'}\|_{H^{-2}(\R^2)}\|P_{\leq N}\phi\|_{H^2(\R^2)}\\&\lesssim N^3\sup_{|t|\leq T}\|\psi^\epsilon-\psi^{\epsilon'}\|_{H^{-2}(\R^2)}\|\phi\|_{H^{-1}(\R^2)}
        \end{align*}
        For the second term, we simply bound
        \[
        \sup_{|t|\leq T}|\langle\psi^\epsilon-\psi^{\epsilon'},P_{>N}\phi\rangle|\lesssim A\|P_{>N}\phi\|_{H^{-1}(\R^2)}.
        \] 
        Therefore, for any dyadic \(N\), we may take \(\epsilon,\epsilon'\to0\) to get
        \[
        \lim_{\epsilon,\epsilon'\to0}\sup_{|t|\leq T}|\langle\psi^\epsilon-\psi^{\epsilon'},\phi\rangle|\lesssim A\|P_{>N}\phi\|_{H^{-1}(\R^2)}.
        \]
        Taking \(N\to\infty\), we conclude that
        \begin{equation}\label{weakconvergence}
             \lim_{\epsilon,\epsilon'\to0}\sup_{|t|\leq T}|\langle\psi^\epsilon-\psi^{\epsilon'},\phi\rangle|=0.
         \end{equation}

        At this point we have established sufficiently strong convergence to prove that \(\psi\) solves the Duhamel formulation \eqref{eq:Dulm} of \eqref{nls-y}. Precisely, for any \(\phi\in \Schwartz\) and all $|t| < T$, we have
        
    \begin{equation}\label{identification}
    \begin{aligned}
    &\lim_{\epsilon\to0} \left\langle \phi,\mu \int_0^t e^{i(t-s)\Delta}
    \bigl[\rho_\Sigma^\epsilon |\psi^\epsilon(s)|^2\psi^\epsilon(s)\bigr]\,ds \right\rangle \\
    &\quad= \lim_{\epsilon\to0} \mu\int_0^t\int_{\mathbb R^2}
    \overline{e^{-i(t-s)\Delta}\phi(x)}
    \,\rho_\Sigma^\epsilon(x) |\psi^\epsilon(s,x)|^2\psi^\epsilon(s,x)\,dx\,ds \\
    &\quad= \mu\int_0^t\int_\Sigma
    \overline{e^{-i(t-s)\Delta}\phi(y)}
    \,|\psi(s,y)|^2\psi(s,y)\,d\sigma(y)\,ds \\
    &\quad= \left\langle \phi,\mu\int_0^t e^{i(t-s)\Delta}
    \bigl[\delta_\Sigma|\psi(s)|^2\psi(s)\bigr]\,ds \right\rangle .
    \end{aligned}
    \end{equation}
        
        This suffices to prove that for all \(\psi_0\in H^1\) there exists a solution \(\psi\in C([-T,T];H_w^1(\R^2)\cap L_y^\infty L_x^4)\) of \eqref{eq:Dulm}, which is unique by Proposition~\ref{K-estimate}.
        
        To complete the proof, we will show that \begin{equation}\label{normconvergence}
    \lim_{\epsilon\to 0} \sup_{|t|\leq T}\bigl |\|\psi^\epsilon\|_{H^1(\R^2)} - \|\psi\|_{H^1(\R^2)} \bigr| \to 0.
\end{equation}
Combining this with \eqref{weakconvergence} and Radon--Riesz proves that \(\psi^\epsilon\to \psi\in C([-T,T];H^1(\R^2))\).

        For convenience, we denote
        \(V_\epsilon\) as in \eqref{V-epsilon definition} and take
        \[
            V_0(\psi):=\int_{\Sigma}|\psi|^4\,dS.
        \]
        The limiting and regularized energy functionals are then given by
        \[
        E(\psi)=\frac12\|\nabla \psi\|_{L^2}^2+\frac{\mu}{2}V_0(\psi)
        \qtq{
        and
        }
        E_\epsilon(\psi)=\frac12\|\nabla \psi\|_{L^2}^2 +\frac{\mu}{2}V_\epsilon(\psi).
        \]
        Therefore, we have
        \[
        \|\nabla \psi\|_{L^2}^2=2E(\psi)-\mu V_0(\psi)
        \qtq{
        and
        }\|\nabla \psi^\epsilon\|_{L^2}^2 =2E_\epsilon(\psi^\epsilon)-\mu V_\epsilon(\psi^\epsilon).
        \]        
        Applying these identities to \(\psi^\epsilon(t)\) and \(\psi(t)\),
        we obtain
        \begin{align}
        \|\nabla\psi^\epsilon(t)\|_{L^2}^2-\|\nabla\psi(t)\|_{L^2}^2
        =2\Bigl(E_\epsilon(\psi^\epsilon(t))-E(\psi(t))\Bigr)
        -\mu\Bigl(V_\epsilon(\psi^\epsilon(t))-
        V_0(\psi(t))\Bigr).
        \label{eq:gradient-energy-difference}
        \end{align}
        
        We first establish convergence of the initial energies. By definition,
        \[
        E_\epsilon(\psi_0)-E(\psi_0)=\frac{\mu}{2}\Bigl(V_\epsilon(\psi_0)-V_0(\psi_0)\Bigr),
        \]
        and Lemma~\ref{lem:Aepsilon} gives us
        \eq{init energy converges}{
        E_\epsilon(\psi_0)\longrightarrow E(\psi_0)
        \qquad\text{as }\epsilon\to0.
        }
        
        Next, we prove convergence of the potential energy along the
        approximating solutions. We use the decomposition
        \begin{align}\label{ND0}
        \left|V_\epsilon(\psi^\epsilon(t))-
        V_0(\psi(t))\right|
        &\leq
        \left|V_\epsilon(\psi^\epsilon(t))-V_0(\psi^\epsilon(t))
        \right|
        +
        \left|V_0(\psi^\epsilon(t))-V_0(\psi(t))\right|.
        \end{align}
        The first term vanishes uniformly for \(|t|\leq T\) by
        Lemma~\ref{lem:Aepsilon}, together with the uniform \(H^1(\R^2)\)-bound
        on \(\psi^\epsilon\).
        \begin{align*}
        \left|V_0(\psi^\epsilon)-V_0(\psi)\right|&\lesssim\left(\|\psi^\epsilon\|_{L^4(\Sigma)}^3+\|\psi\|_{L^4(\Sigma)}^3\right)\|\psi^\epsilon-\psi\|_{L^4(\Sigma)},
        \end{align*}

        and then apply \eqref{eq:L4Cy} to obtain
        \eq{pot energy converges}{
        \lim_{\epsilon\to0}
        \sup_{|t|\leq T}\left|
        V_\epsilon(\psi^\epsilon(t)) -V_0(\psi(t))\right|=0.
        }

        By combining the convergence of the potential energy with the weak  convergence of \(\psi^\epsilon(t)\to \psi(t)\) in \(H^1(\R^2)\), we obtain   
        \begin{align*}
        M(\psi(t))
        &\leq \liminf_{\epsilon\to0}M(\psi^\epsilon(t))
        =
        M(\psi_0),\\
        &\leq
        \liminf_{\epsilon\to0}\left(
        \frac12\|\nabla\psi^\epsilon(t)\|_{L^2}^2+
        \frac{\mu}{2}V_\epsilon(\psi^\epsilon(t))\right)=\lim_{\epsilon\to0}E_\epsilon(\psi_0)        =E(\psi_0).
        \end{align*}
        These inequalities hold for both signs of \(\mu\), since the
        potential energy converges.        
 
        Lemma~\ref{Conservation-limiting} now yields conservation of mass and energy on
        \([0,T]\); the same argument applies to negative times.
        
        From conservation of energy and \eqref{init energy converges}, we conclude that
        \[
        E_\epsilon(\psi^\epsilon(t))
        =E_\epsilon(\psi_0)
        \longrightarrow
        E(\psi_0)=
        E(\psi(t)),
        \]
        as \(\epsilon \to0\), uniformly for all \(|t|\leq T\). Combining this with \eqref{eq:gradient-energy-difference}, conservation of mass, and convergence of the potential energy \eqref{pot energy converges} yields \eqref{normconvergence}, completing the proof.
\end{proof}

Now we have proved local existence and uniqueness in \(H^1(\R^2)\), an identical bootstrap argument to Lemma~\ref{lem:uniform-regularized-control} proves the following:

\begin{cor}
    Let  \(n=2\), there exists a unique strong global solution \(\psi\in C(\R;\,H^1(\R^2))\) to \eqref{nls-y} with initial data $\psi_0\in H^1(\R^2)$ if either of the following two conditions holds:

(i)  \eqref{nls-y} is defocusing;

(ii) \(\|\psi_0\|_{H^1(\R^2)}\) is sufficiently small.

This solution conserves both the mass and the energy.

\end{cor}

To complete the proof of Theorem~\ref{thrm:GWPy} in the case \(\Sigma = \{y=0\}\), it remains to prove the solution map is continuous:
    
\begin{prop}\label{Ctsdependence}
If \(\psi_{n,0}\to\psi_0\) is a convergent sequence of \(H^1(\R^2)\) initial data with corresponding solutions \(\psi_n,\psi\in C(\R;H^1(\R^2))\) of \eqref{nls-y} in either the defocusing case or focusing case with sufficiently small data, then for any time \(0<T<\infty\) we have
\[
\lim_{n\to\infty}\sup_{|t|\leq T}\|\psi_n(t)-\psi(t)\|_{H^1(\R^2)}=0.
\]
\end{prop}

\begin{proof}

By conservation of mass and energy and the trace bound
\(V_0(u)\lesssim\|u\|_{H^1(\R^2)}^4\), the argument of
Lemma~\ref{lem:uniform-regularized-control}, with \(V_\epsilon\)
replaced by \(V_0\), shows that there exists \(0<A<\infty\)
such that
\[\sup_{n\geq 1}\sup_{|t|\leq T} \|\psi_n \|_{H^1(\R^2)} + \sup_{|t|\leq T} \|\psi\|_{H^1(\R^2)} \leq A.\]
We may then apply Proposition~\ref{K-estimate} to obtain
\[\lim_{n\to \infty } \sup_{|t| \leq T} \|\psi_{n}(t) - \psi(t)\|_{H^{-1}(\R^2 )} = 0.\]

We may upgrade this convergence by interpolation to $H^{\frac{3}{4}}$, 
\[ \|\psi_{n}(t) - \psi(t)\|_{H^{3/4}(\R^2 )} \lesssim \|\psi_{n}(t) - \psi(t)\|^{1/8}_{H^{-1}(\R^2)} \|\psi_{n}(t) - \psi(t)\|^{7/8}_{H^1(\R^2)}\lesssim \|\psi_{n}(t) - \psi(t)\|^{1/8}_{H^{-1}(\R^2)} A^{7/8},\]
and hence
\begin{equation}\label{H34}
\lim_{n\to \infty } \sup_{|t| \leq T} \|\psi_{n}(t) - \psi(t)\|_{H^{3/4}(\R^2 )} = 0.
\end{equation}
By Sobolev embedding and the trace theorem, we then have
\begin{align*}
\lim_{n\to \infty } \sup_{|t| \leq T} \|\psi_{n}(t) - \psi(t)\|_{L^4(\Sigma)} \lesssim \lim_{n\to \infty } \sup_{|t| \leq T} \|\psi_{n}(t) - \psi(t)\|_{H^{3/4}(\R^2)} = 0.
\end{align*}

Next, we estimate
\begin{align*}
\sup_{|t|\leq T}\left|\|\psi_n(t)\|_{H^1(\R^2)}^2-\|\psi(t)\|_{H^1(\R^2)}^2\right|&\leq
2\left|M[\psi_{n,0}]-M[\psi_0]\right|+2\left|E[\psi_{n,0}]-E[\psi_0]\right|
\\
&+\sup_{|t|\leq T}\left|\int_{\Sigma}|\psi_n(t)|^4\,dS-
\int_{\Sigma}|\psi(t)|^4\,dS\right|
\end{align*}
By H\"older's inequality, we have
\[\label{dependh1cvg}\sup_{t\leq T}\left|\int_{\Sigma}|\psi_n(t)|^4\,dS-
\int_{\Sigma}|\psi(t)|^4\,dS\right| \lesssim \left( \sup_{t\leq T}\|\psi_n(t)\|^{3}_{L^4(\Sigma))} + \sup_{t\leq T} \|\psi(t)\|^{3}_{L^4(\Sigma))}\right) \sup_{t\leq T}\|\psi_n(t) - \psi(t)\|_{L^4(\Sigma)},\]
and hence
\[\lim_{n\to\infty}\sup_{|t|\leq T}\left|\|\psi_n(t)\|_{H^1(\R^2)}^2-\|\psi(t)\|_{H^1(\R^2)}^2\right|= 0.\]

To complete the proof, we observe that
\[
\|\psi_n(t) - \psi(t)\|_{H^1(\R^2)}^2
=
\|\psi_n(t)\|_{H^1(\R^2)}^2-
\|\psi(t)\|_{H^1(\R^2)}^2-2\operatorname{Re}\langle \psi(t),\psi_n(t) - \psi(t)\rangle_{H^1(\R^2)}.
\]
Consequently, it suffices to show that
\[\lim_{n\to\infty}\sup_{|t|\leq T} \left|\langle (1-\Delta )\psi(t) , \psi_n(t) - \psi(t)\rangle \right|=0.\]

To prove this, we observe that \[\{(1-\Delta)\psi(t): |t| \leq T\} \subset H^{-1}(\R^2).\] is a continuous image of compact set and therefore compact in $H^{-1}(\R^2)$. In particular,
\[
\lim_{N\to\infty}\sup_{|t|\leq T}\|P_{>N}(1-\Delta)\psi(t)\|_{H^{-1}(\R^2)} = 0.
\]
We then have
\begin{align*}
\left|\langle (1-\Delta )\psi(t) , \psi_n(t) - \psi(t)\rangle \right| & \leq \left|\langle (1-\Delta )P_{\leq N}\psi(t) , \psi_n(t) - \psi(t)\rangle \right| + \left|\langle (1-\Delta )P_{>N}\psi(t) , \psi_n(t) - \psi(t)\rangle \right|\\
&\lesssim N^2\|\psi(t)\|_{H^1(\R^2)}\|\psi_n(t) - \psi(t)\|_{H^{-1}(\R^2)}\\
&\quad + \|P_{>N}(1-\Delta)\psi(t)\|_{H^{-1}(\R^2)}\|\psi_n(t) - \psi(t)\|_{H^1(\R^2)}.
\end{align*}
Taking the supremum in \(t\), passing to the limit as \(n\to\infty\) and then as \(N\to \infty\) completes the proof.

\end{proof}

\section{Global Wellposedness for \eqref{nls-S} in \(H^1(\R^2)\)}\label{sec:radial}
In this section we prove Theorem~\ref{thrm:GWPy} in the case of the sphere, \eqref{nls-S}.

In this section, when we write \(L_\theta^p\) or \(H_\theta^s\), we mean the spaces generated by measure \(d\theta\) with \(\theta\in \T = \R/2\pi\Z\). We write \(L_{r}^p\) and \(H_{r}^s\) to denote the \(1d\) Lebesgue and Sobolev spaces with \(r\in \R\). We take \(L_{r,+}^p\) and \(H_{r,+}^s\) to denote the corresponding \(1d\) Lebesgue and Sobolev spaces with \(r\in (0,\infty)\). Finally, we take \(L^p(rdr)\) and \(H^s(rdr)\) to denote the Lebesgue and Sobolev spaces for radial \(2d\) functions, with norms
\begin{alignat*}{3}
\|f\|_{L^p(rdr)}&:=\left(\int_0^\infty |f|^p\,r\,dr\right)^{\frac 1 p} \quad &&\text{for }1\leq p <\infty,\\
\|f\|_{H^s(rdr)}&:=\left(\int_0^\infty (1+\rho^2)^s|\tilde f(\rho)|^2\,\rho\,d\rho\right)^{\frac 1 2}\quad&&\text{for \(s\in \R\)},
\end{alignat*}
where \(\tilde f(\rho)\) is the Hankel transform of \(f\) defined as
\[
\tilde f(\rho):=\mathcal H_0[f](\rho):=\int_0^\infty f(r)J_0(r\rho)\,r\,dr,
\]
where \(J_0\) is the Bessel function of the first kind.
We note that \(L^\infty(rdr)=L^\infty_{r,+}\).

Below we define Littlewood--Paley projections in polar coordinate \(P^r_{\geq N}\) and \(P^\theta_{\geq N}\) for a compactly supported function $f\colon(0,\infty)\times \T\to \C$.

Given dyadic \(N\geq 1\), let $E$ be the zero extension operator given by
\[
 Ef(r,\theta)=\begin{cases}
    f(r,\theta) &r\geq 0,\\
    0 & r<0.
\end{cases}
\]
We then define 
\[
\widehat{P^r_{\leq N}Ef}(\rho,\theta):=\varphi\left(\frac{\rho}{N}\right)\widehat{Ef}(\rho,\theta),
\]
\(\varphi\in C_c^\infty(-2,2)\) is a non-negative bump function that is identically \(1\) on \([-1,1]\) and
\[
\widehat{Ef}(\rho,\theta) = \frac1{\sqrt{2\pi}}\int_\R Ef(r,\theta)e^{-ir\rho}\,dr,
\]
is the usual Fourier transform on \(\R\). In the angular variable, we define
\[P_{\leq N }^\theta f(r,\theta) = \sum_{ |n|\leq N } \frac 1 {2\pi}\int_{0}^{2\pi} f(r,\omega)e^{in(\theta - \omega)}{d\omega}.\]
We denote \(P^r_{>N}:=1-P^r_{\leq N}\), and \(P^\theta_{>N}:=1-P^\theta_{\leq N}\).

Associated to these estimates, we have the following Bernstein inequalities, with proof as in \cite{Tao2006NonlinearDispersive}:
\begin{lem}\label{Bernstein_r} For all compactly supported \(f\colon (0,\infty)\times \T\to \C\) we have the estimates,
\begin{align}
    &\|P^r_{\leq N}Ef\|_{L^\infty_r}\lesssim N^{-s+\frac 1 2}\|Ef\|_{H^s_r}=N^{-s+\frac 1 2}\|f\|_{H^s_{r,+}} \qquad s<\frac 1 2,\label{bernie down r}\\
    &\|P^r_{> N}Ef\|_{L^\infty_r}\lesssim N^{-s+\frac 1 2}\|Ef\|_{H^s_r}=N^{-s+\frac 1 2}\|f\|_{H^s_{r,+}} \qquad s>\frac 1 2,\label{bernie up r}\\
    &\|P^\theta_{\leq N}f\|_{L^4_\theta}\lesssim N^{-s+\frac 1 4}\|P^\theta_{\leq N}f\|_{H^s_\theta}\lesssim N^{-s+\frac 1 4}\|f\|_{H^s_\theta} \qquad s\leq\frac 1 4,\label{bernie down theta}\\
    &\|P^\theta_{> N}f\|_{L^4_\theta}\lesssim N^{-s+\frac 1 4}\|P^\theta_{> N}f\|_{H^s_\theta}\lesssim N^{-s+\frac 1 4}\|f\|_{H^s_\theta} \qquad s\geq\frac 1 4.\label{bernie up theta}
\end{align}
\end{lem}

At the heart of our argument is the following dispersive estimate. This gives an improved short time bound over the classical dispersive estimate \eqref{dispersive} by trading spatial weights for decay. It crucially relies on a uniform bound for Bessel functions due to Landau \cite{Landou2000Bessel}.
\begin{lem}\label{l:dispersive estimate in section 4}
For all \(t\neq 0\), we have the weighted dispersive estimates
    \begin{align}
    \|r^{\frac 16}e^{it\Delta}f\|_{C_{0,r,+}L^4_\theta }&\lesssim |t|^{-\frac 56}\|r^{-\frac 16}f\|_{ L^1(rdr)L^{\frac 4 3}_\theta},\label{dispersive-S-C}\\
    \|r^{\frac 16}e^{it\Delta}f\|_{L_{r,+}^\infty L^4_\theta }&\lesssim |t|^{-\frac 56}\|r^{\frac 56}f\|_{ \M_{r,+}L^{\frac 4 3}_\theta},\label{dispersive-S-M}
    \end{align}
    where \(\M_{r,+}\) denotes the space of finite signed measures on \((0,\infty)\), i.e., the dual of \(C_{0,r,+}\) with respect to the usual pairing on \((0,\infty)\).
\end{lem}
\begin{proof}
We assume that \(t>0\). The corresponding proof for \(t<0\) follows by symmetry. We first aim to show that
\begin{align}\label{09-15-22-57}
    \|r^{\frac 13}e^{it\Delta}f\|_{L^\infty_{r,+} L^2_\theta}\lesssim |t|^{-\frac 23}\|r^{-\frac 13}f\|_{ L^1(rdr)L^2_\theta}.
\end{align}

    Let us first decompose
    \[
    f(s,\phi)=\sum_{n\in \Z}f_n(s)e^{in\phi},
    \]
    where
\[
f_n(s) := \frac1{2\pi}\int_0^{2\pi}f(s,\theta)e^{-in\theta}\,d\theta.
\]
    By the Hankel transform, we then get
    \[
    (e^{it\Delta}f)_n (r) =\frac{i^{-n}}{2it}\int_0^\infty f_n(s)J_n\left(\frac{rs}{2t}\right)e^{i\frac{r^2+s^2}{4t}}sds.
    \]
    Multiplying by $r^{\frac 13}$, we rewrite the integral as
    \[
    r^{\frac 13}(e^{it\Delta}f)_n(r) = \frac{i^{-n}}{2it} (2t)^{\frac 13} \int_0^\infty \left[ \left(\frac{rs}{2t}\right)^{\frac 13} J_n\left(\frac{rs}{2t}\right) \right] e^{i\frac{r^2+s^2}{4t}} f_n(s) s^{\frac 23} ds.
    \]
    Let $G_n(x) = |x|^{\frac 13} J_n(x)$. We recall the Landau bound \cite{Landou2000Bessel} given by
    \[
    |J_n(x)|\lesssim |x|^{-\frac 1 3}, \quad\text{uniformly for all \(x>0\) and \(n\in \N\)},
    \]
    which ensures $|G_n(x)| \lesssim 1$ uniformly for all $n \in \N$ and $x > 0$. Furthermore, from the asymptotics of the Bessel function\cite{NIST:DLMF}, there exist constants \(A_n,B_n,\alpha_n\) so that:
    \begin{itemize}
        \item As $x \to 0$, $J_n(x) \sim A_nx^n$, which implies $\lim_{x \to 0} G_n(x) = 0$ for all $n \ge 0$.
        \item As $x \to \infty$, $J_n(x) \sim B_nx^{-\frac 12} \cos(x - \alpha_n)$, which implies $\lim_{x \to \infty} G_n(x) = 0$.
    \end{itemize}
    Thus, $G_n \in C_0(\R^+)$. As a result, for fixed $t>0$ and $s > 0$, the mapping $r \mapsto \left(\frac{rs}{2t}\right)^{\frac 1 3}J_n\left(\frac{rs}{2t}\right)$  \(\in C_0(\R^+)\). Assuming  $s^{-1/3}f_n \in L^1(sds)$, the integrand for $r^{1/3}(e^{it\Delta}f)_n$ is uniformly dominated (up to a constant) by the integrable function $ |s^{-1/3}f_n(s)|s$. For any fixed $t > 0$ and $s > 0$, taking the limit as $r \to 0$ or $r \to \infty$ drives $\left(\frac{rs}{2t}\right)^{\frac 1 3}J_n\left(\frac{rs}{2t}\right) \to 0$. By the Dominated Convergence Theorem, we obtain \(r^{\frac 13}(e^{it\Delta}f)_n(r)\in C_0(\R^+)\).
    
    Applying the Landau bound to the integrand, we also recover the pointwise bound
    \[
    |r^{\frac 13} (e^{it\Delta}f)_n(r)|\lesssim |t|^{-\frac 23}\|s^{-\frac 13}f_n\|_{L^1(sds)}.
    \]
    Taking the $l^2_n$ norm on both sides gives us
    \[
    \|r^{\frac 13}(e^{it\Delta}f)_n\|_{l^2_n}\lesssim |t|^{-\frac 23}\|r^{-\frac 13}f_n\|_{l^2_nL^1(rdr)}.
    \]
    Then from Minkowski's inequality and Plancherel's Theorem, we obtain \eqref{09-15-22-57}.

    Then we use the Riesz--Thorin interpolation theorem for mixed norms \cite{Blozinski1981MultivariateRA} to interpolate with the original dispersive estimate \eqref{dispersive}, which can be written as
\[
\|e^{it\Delta} f\|_{L^\infty_{r,+}L^\infty_\theta}\lesssim |t|^{-1}\|f\|_{L^1(rdr)L^1_\theta },
\]
we get the dispersive estimate
\begin{align*}
    \|r^{\frac 1 6}e^{it\Delta}f\|_{ L^\infty_{r,+}L^4_\theta}\lesssim |t|^{-\frac 56}\| r^{-\frac 1 6}f\|_{L^1(rdr)L^{\frac 4 3}_\theta },
\end{align*}
    To conclude that \(r^{\frac 1 6}e^{it\Delta}f\in C_{0,r,+}L^4_\theta\), we denote by \(\mathcal D\) the dense subspace of \(\{f:r^{-\frac 1 6}f\in L^1(rdr)L^{\frac 4 3}_\theta\}\) such that only finitely many \(f_n\neq 0\). Then for any \(f\in \mathcal D\), \(r^{\frac 1 3}e^{it\Delta}f\in C_{0,r,+}L^4_\theta\) and the estimate \eqref{dispersive-S-C} follows from the fact that \(C_0(\R^+)\) is a closed subspace of \(L^\infty(\R^+)\). 

To prove \eqref{dispersive-S-M}, we first write \eqref{dispersive-S-C} in the form
\[
\|r^{\frac 16}e^{it\Delta}f\|_{C_{0,r,+}L^4_\theta }\lesssim |t|^{-\frac 56}\|r^{\frac 56}f\|_{ L_{r,+}^1L^{\frac 4 3}_\theta},
\]
and then observe that the dual of \(S_t:=r^{\frac 1 6}e^{it\Delta}r^{-\frac5 6}\) with respect to the pairing \(\<f,g\> = \int_0^{2\pi}\int_0^\infty \bar f g\,dr\,d\theta\) is \((S_t)^*=r^{\frac 1 6}e^{-it\Delta}r^{-\frac 5 6}\).
\end{proof}

For the remainder of this article we take \(\chi\) to be a smooth bump function supported on \((\frac 1 2,2)\) and identically \(1\) on \([\frac34,\frac54]\). Applying dispersive estimates \eqref{dispersive-S-C} and \eqref{dispersive-S-M} to \(\chi f\),
we get
\begin{align}\label{dispersive-S-chi-C}
    \|\chi e^{it\Delta}\chi f\|_{ C_{0,r,+}L^4_\theta}&\lesssim |t|^{-\frac 56}\|\chi f\|_{L^1_{r,+}L^{\frac 4 3}_\theta },
\\\label{dispersive-S-chi-M}
    \|\chi e^{it\Delta}\chi f\|_{ L^\infty_{r,+}L^4_\theta}&\lesssim  |t|^{-\frac 56}\|\chi f\|_{\M_{r,+}L^{\frac 4 3}_\theta }.
\end{align}
Moreover, with this bump function \(\chi\), it is clear that
\begin{alignat*}{3}
\|\chi f\|_{H^1(rdr)} &\approx \|\chi f\|_{H^1_r}&&=\|\chi f\|_{H^1_{r,+}},
\\
\|\chi f\|_{L^2(rdr)}&\approx \|\chi f\|_{L^2_r}&&=\|\chi f\|_{L^2_{r,+}}.
\end{alignat*}
We may then use complex interpolation to deduce that
\begin{equation}\label{rdr=dr}
    \|\chi f\|_{H^s(rdr)}\approx \|\chi f\|_{H^s_r} =\|\chi f\|_{H^s_{r,+}} \quad \text{for all }|s|\leq 1.
\end{equation}

By applying the Sobolev embedding \eqref{1-d GNinfty} and the estimate \eqref{1-d GN4} to \(\chi f\), we obtain the following anisotropic Sobolev embedding, which we will use several times in this section: for any \(\epsilon\in (0,\frac 1 4]\), we have
 \begin{align}\label{eq:anisotrophic-r}
    \|\chi f\|_{L^\infty_{r,+}L^4_\theta}&\lesssim \|\langle \p_\theta\rangle^{\frac14}\chi f\|_{L^\infty_{r,+}L^2_\theta}\\
    &\lesssim \|\langle \p_\theta\rangle^{\frac14}\chi f\|_{L^2_\theta L^\infty_{r,+}}\notag\\
    &\lesssim \| \langle \p_r \rangle^{\frac 1 2+\epsilon}\langle \p_\theta\rangle^{\frac 14}\chi f\|_{L^2_\theta L^2_{r,+}}\notag\\
    &\lesssim \|\chi f\|_{L_{r,+}^2L_\theta^2}+\|\p_r\chi f\|_{L_{r,+}^2L_\theta^2}+\|\p_\theta\chi f\|_{L_{r,+}^2L_\theta^2}\notag\\
    &\lesssim \|\chi f\|_{H^1(\R^2)},\notag
 \end{align}
 where the final inequality uses the support of \(\chi\).

By duality, we obtain
\begin{equation}\label{eq:anisotrophic-r-dual}
       \|\chi f\|_{H^{-1}(\R^2)}\lesssim \|\chi f\|_{L_r^1L_\theta^{\frac43}}\approx \|\chi f\|_{L^1(rdr)L^\frac{4}{3}_\theta}\approx \|\chi f\|_{L^1_{r,+}L^\frac{4}{3}_\theta}.
 \end{equation}

While \(H^1(\R^2)\) is not an algebra, our next Lemma proves that, when we work with radial functions supported on an annulus, we can essentially reduce matters to \(1d\) and recover algebra estimates:

\begin{lem}\label{lem:fakealgebraic}
    If \(\chi f=f\), and \(f(x)=f(|x|)=f(r)\), then
    \[
    \|fghk\|_{L^\frac 43_\theta H^{-\frac 3 4}_{r,+}}\lesssim\|f\|_{H^{-\frac 3 4}_{r,+}}\|g\|_{H^1(\R^2)}\|h\|_{H^1(\R^2)}\|k\|_{H^1(\R^2)}
    \]
\end{lem}
\begin{proof}
    We proceed by duality. Let \(\phi\) be a test function and let \(\tilde\chi \chi=\chi\) where \(\tilde\chi\) is a  smooth bump function on \((\frac 1 4,4)\). Then, by \eqref{rdr=dr}, \eqref{eq:algebraic}, and \eqref{eq:anisotrophic-r}, we have
    \begin{align*}
        \int_{\R^2} fghk\phi dx&=\int_0^{2\pi}\int_0^\infty(\chi f)(\tilde\chi g)(\tilde\chi h)(\tilde\chi k)\phi\,r\,dr\,d\theta\\
        &\lesssim\|\chi f\|_{H^{-\frac 3 4}(rdr)}\|(\tilde\chi g)(\tilde\chi h)(\tilde\chi k) \phi\|_{L^1_\theta H^\frac 3 4(rdr) }\\
        &\approx \|\chi f\|_{H^{-\frac 3 4}_{r,+}}\|(\tilde\chi g)(\tilde\chi h)(\tilde\chi k)\phi\|_{L^1_\theta H^\frac 3 4_{r,+} }\\
        &\lesssim \|\chi f\|_{H^{-\frac 3 4}_{r,+}}\|\tilde\chi g\|_{L^4_\theta H^\frac 3 4_{r,+} }\|\tilde\chi h\|_{L^4_\theta H^\frac 3 4_{r,+}}\|\tilde\chi k\|_{L^4_\theta H^\frac 3 4_{r,+} }\|\phi\|_{L^4_\theta H^\frac 3 4_{r,+}}\\
        &\lesssim\|f\|_{H^{-\frac 3 4}_{r,+}}\|g\|_{H^1(\R^2)}\|h\|_{H^1(\R^2)}\|k\|_{H^1(\R^2)}\| \phi\|_{L^4_\theta H^\frac 3 4_{r,+}}.
    \end{align*}
    The desired estimate now follows by duality.
\end{proof}

The following Lemma bridges the gap between the spaces \(L^2_\theta H^{s}_{r,+}\) and \(H^s(\R^2)\) provided that \(f\) is localized around \(r=1\) and contains only low angular frequencies.
\begin{lem}\label{mixed-full}
    If \(f=P_{\leq N}^\theta f\) and \(\chi f=f\), then for \(0\leq s\leq 1\) we have
    \[
     \|f\|_{L^2_\theta H^{s}_{r,+}} \lesssim \|f\|_{H^{s}(\mathbb{R}^2)} \lesssim N^{s}\|f\|_{L^2_\theta H^{s}_{r,+}},
    \]
    and by duality, for \(-1\leq s<0\), we have
    \[
     N^{s}\|f\|_{L^2_\theta H^{s}_{r,+}} \lesssim \|f\|_{H^{s}(\mathbb{R}^2)} \lesssim \|f\|_{L^2_\theta H^{s}_{r,+}}.
    \]
\end{lem}
\begin{proof}
    We first consider the case \(s=1\).
    We note that
    \[
    \|f\|_{L^2_\theta H^{1}_{r,+}}\approx \|f\|_{L^2(\R^2)}+\|\p_rf\|_{L^2(\R^2)},
    \]
    and since \(\chi f = f\) and \(f = P_{\leq N}^\theta f\),
    \begin{align*}
        \|f\|_{H^1(\mathbb{R}^2)}^2
        &\approx\|f\|_{L^2(\R^2)}^2 + \|\partial_r f\|_{L^2(\R^2)}^2 + \left\| \partial_\theta f \right\|_{L^2(\R^2)}^2\lesssim N^2\|f\|_{L^2(\R^2)}^2 + \|\partial_r f\|_{L^2(\R^2)}^2.
    \end{align*}
    As a result,
    \[
     \|f\|_{L^2_\theta H^{1}_{r,+}} \lesssim \|f\|_{H^{1}(\mathbb{R}^2)} \lesssim N\|f\|_{L^2_\theta H^{s}_{r,+}}.
    \]

    The trivial case \(s=0\) gives us
    \[
    \|f\|_{L^2_\theta L^2_{r,+}} = \|f\|_{L^2(\mathbb{R}^2)},
    \]
   and desired estimates for \(s \in (0, 1)\) are concluded using complex interpolation theorem. The range \(-1 \leq s < 0\) then follows by duality.
\end{proof}

With all these preparations in hand, we can go ahead and prove our well-posedness result. Again, we start by proving a Lipschitz estimate for the solution map of \eqref{nls-S}, which yields uniqueness as an immediate corollary:

\begin{prop}\label{p:10-15-2}
    Suppose that for some \(T>0\) there are two solutions \(\psi,\phi\in C([-T,T];H_w^1(\R^2))\) of the Duhamel formulation \eqref{eq:Dulm} of \eqref{nls-S} so that \(\chi\psi,\chi\phi\in C([-T,T];L_{r,+}^\infty L_\theta^4)\). Then for all \(|t|\leq T\) we have the estimate
    \begin{equation}\label{eq:Gronwell-S}
        \|\psi(t)-\phi(t)\|_{H^{-1}(\R^2)} \lesssim (1+K|t|e^{CK^{12}t^2}) \|\psi(0)-\phi(0)\|_{H^1(\R^2)},
    \end{equation}
    where \(C>0\) is some absolute constant and (using \eqref{eq:anisotrophic-r})
    \[
    K:=\sup_{|t|\leq T}\Bigl\{\|\chi\psi(s)\|_{ L^\infty_{r,+}L^4_\theta}^2+\|\chi\phi(s)\|^2_{ L^\infty_{r,+}L^4_\theta}\Bigr\}\lesssim \sup_{|t|\leq T}\Bigl\{\|\psi(s)\|_{H^1(\R^2)}^2+\|\phi(s)\|^2_{H^1(\R^2)}\Bigr\}.
    \]
\end{prop}
\begin{proof} 
By the Duhamel formula, we have
    \[
    \psi(t)-\phi(t)=e^{it\Delta}(\psi(0)-\phi(0))-i2\mu\int_0^t e^{i(t-s)\Delta}\delta_{r=1}\left[|\psi(s)|^2\psi(s)-|\phi(s)|^2\phi(s)\right]ds
    \]
    Then, for \(t>0\) we estimate
    \[
    \|\chi(\psi(t)-\phi(t))\|_{L^\infty_{r,+}L^4_\theta }\lesssim \|\chi e^{it\Delta}(\psi(0)-\phi(0))\|_{L^\infty_{r,+}L^4_\theta } +\left\|\int_0^t \chi e^{i(t-s)\Delta}\chi\delta_{r=1}\left[|\psi(s)|^2\psi(s)-|\phi(s)|^2\phi(s)\right]ds\right\|_{L^\infty_{r,+}L^4_\theta }
    \]
    For the first term, we apply \eqref{eq:anisotrophic-r} to bound
    \begin{align*}
        \|\chi e^{it\Delta}(\psi(0)-\phi(0))\|_{L^\infty_{r,+}L^4_\theta }
        \lesssim \|e^{it\Delta}(\psi(0)-\phi(0))\|_{H^1(\R^2)}\lesssim \|\psi(0)-\phi(0)\|_{H^1(\R^2)}
    \end{align*}
    
    For the second term, we apply \eqref{dispersive-S-chi-M} to obtain
    \begin{align*}
        &\left\|\int_0^t \chi e^{i(t-s)\Delta}\chi \delta_{r=1}\left[|\psi(s)|^2\psi(s)-|\phi(s)|^2\phi(s)\right]ds\right\|_{L^\infty_{r,+}L^4_\theta }\\
        &\quad \lesssim  \int_0^t|t-s|^{-\frac 56}\|\chi^3 \delta_{r=1}\left[|\psi(s)|^2\psi(s)-|\phi(s)|^2\phi(s)\right]\|_{\M_{r,+}L^\frac{4}{3}_\theta }ds\\
        &\quad \lesssim K\int_0^t|t-s|^{-\frac 56}\|\chi(\psi(s)-\phi(s))\|_{L^\infty_{r,+}L^4_\theta }ds\\
        &\quad \lesssim K\left(\int_0^t|t-s|^{-\frac {10} {11}}ds\right)^\frac{11}{12}\left(\int_0^t\|\chi(\psi(s)-\phi(s))\|_{L^\infty_{r,+}L^4_\theta }^{12}ds\right)^\frac {1}{12}\\
         &\quad \lesssim Kt^{\frac 1{12}}\left(\int_0^t\|\chi(\psi(s)-\phi(s))\|_{L^\infty_{r,+}L^4_\theta }^{12}ds\right)^\frac {1}{12},
    \end{align*}
    where we have used that \(\chi(1) = 1 = \chi(1)^3\).
     Taking the 12\textsuperscript{th} power gives us
     \[
     \|\chi(\psi(t)-\phi(t))\|_{L^\infty_{r,+}L^4_\theta }^{12}\lesssim\|\psi(0)-\phi(0)\|_{H^1(\R^2)}^{12}+K^{12}t\int_0^t\|\chi(\psi(s)-\phi(s))\|^{12}_{L^\infty_{r,+}L^4_\theta }\,ds.
     \]
     Applying the integral form of Gronwall's inequality, we then have
    \[
      \|\chi(\psi(t)-\phi(t))\|_{L^\infty_{r,+}L^4_\theta }\lesssim e^{CK^{12}t^2}\|\psi(0)-\phi(0)\|_{H^1(\R^2)}.
     \]
     Moreover, for any test function \(f\in H^1(\R^2)\) we have
    \begin{align*}
        \langle\delta_{r=1}\left(|\psi|^2\psi-|\phi|^2\phi\right),f\rangle&=\int_0^{2\pi}\overline{\left(|\psi(1,\theta)|^2\psi(1,\theta)-|\phi(1,\theta)|^2\phi(1,\theta)\right)}f(1,\theta)d\theta\\
        &\lesssim (\|\psi(1,\theta)\|_{L^4_\theta}^2+\|\phi(1,\theta)\|_{L^4_\theta}^2)\|\psi(1,\theta)-\phi(1,\theta)\|_{L_\theta^4}\|f(1,\theta)\|_{L^4_\theta}\\
        &\lesssim K\|\chi(\psi-\phi)\|_{L^\infty_{r,+}L^4_\theta}\|\chi f\|_{L^\infty_{r,+}L^4_\theta}\\
        &\lesssim K\|\chi(\psi-\phi)\|_{L^\infty_{r,+}L^4_\theta}\|f\|_{H^1(\R^2)}.
    \end{align*}
    By duality, we conclude that
    \[
    \bigl\|\delta_{r=1}\left(|\psi|^2\psi-|\phi|^2\phi\right)\bigr\|_{H^{-1}(\R^2)}\lesssim K \|\chi(\psi-\phi)\|_{L^\infty_{r,+}L^4_\theta}.
    \]
     As a consequence,
    \begin{align*}
        \|\psi(t)-\phi(t)\|_{H^{-1}(\R^2)}&\lesssim \|\psi(0)-\phi(0)\|_{H^{-1}(\R^2)}+\int_0^t\|\delta_{r=1}(|\psi(s)|^2\psi(s)-|\phi(s)|^2\phi(s))\|_{H^{-1}(\R^2)}ds\\
        &\lesssim \|\psi(0)-\phi(0)\|_{H^{-1}(\R^2)}+K\int_0^t\|\chi(\psi(s)-\phi(s))\|_{L^\infty_rL^4_\theta }ds\\
        &\lesssim (1+Kte^{CK^{12}t^2}) \|\psi(0)-\phi(0)\|_{H^1(\R^2)},
    \end{align*}
    which completes the proof of \eqref{eq:Gronwell-S}.
\end{proof}

Next, we turn to the problem of existence. We recall that the ambient mollifier \(\rho\) in \eqref{mollification} is radial and hence the mollification \(\rho_{\bbS^1}^\epsilon\) is  also radial. We use the same notation \(\rho^\epsilon_{\bbS^1}(r)\) for its radial profile.

We now consider the regularized problem
\[i\partial_t\psi^\epsilon+\Delta\psi^\epsilon
=2\mu\rho^\epsilon_{\bbS^1}(r)|\psi^\epsilon|^2\psi^\epsilon,\qquad r=|x|,
\] which we recall is globally well-posed in the defocusing case or in the focusing case when \(\|\psi_0\|_{H^1(\R^2)}\) is sufficiently small.

We wish to prove that \(\psi^\epsilon\) converges as \(\epsilon \to 0\), which will give us existence of local-in-time solutions to \eqref{nls-S}.

An identical argument to the line establishes the analog of \eqref{V4estimate}: taking \(V_\epsilon\) to be defined as in \eqref{V-epsilon definition} with \(\Sigma = \bbS^1\) we have
\[
0\leq V_\epsilon(\psi^\epsilon)
\lesssim \|u\|_{H^1(\R^2)}^4,
\]
uniformly in \(\epsilon\).
Lemma~\ref{lem:uniform-regularized-control} therefore gives
global regularized solutions and the uniform bound
\eq{oct 13}{
A: = \sup_{0<\epsilon\leq\frac14}\sup_{t\in\R}
\|\psi^\epsilon(t)\|_{H^1(\R^2)}<\infty
}
in the defocusing case and in the focusing case for
sufficiently small initial data.

\begin{prop}\label{GlobalS1}
    Given \(\psi_0\in H^1(\R^2)\), there exists \(T=T(\|\psi_0\|_{H^1(\R^2)})>0\) and a solution in \(C([-T,T];H^1(\R^2))\) to the Duhamel formulation \eqref{eq:Dulm} of \eqref{nls-S} in the defocusing case, or in the focusing case when \(\|\psi_0\|_{H^1(\R^2)}\) is sufficiently small . Moreover, if \(\psi^\epsilon\) is a family of solutions to \eqref{nls-epsilon} with \(\psi^\epsilon(0)=\psi_0\in H^1(\R^2)\), then
    \[
    \lim_{\epsilon\to 0}\sup_{|t|\leq T}\|\psi^{\epsilon}(t)-\psi(t)\|_{H^1(\R^2)}=0.
    \]
\end{prop}

\begin{proof}
    Taking \(0<\epsilon<\epsilon'\leq \frac14\), we use the Duhamel formulation \eqref{eq:Dulm} to write
    \[
    \psi^{\epsilon}-\psi^{\epsilon'}=e^{it\Delta}(\psi^\epsilon(0)-\psi^{\epsilon'}(0))- i2\mu\int_0^t e^{i(t-s)\Delta}\left(F_1(s)-F_2(s)\right)\,ds,
    \]
    where
    \[
F_1:=\rho^{\epsilon}_{\bbS^1}\left(|\psi^{\epsilon}|^2\psi^{\epsilon}-|\psi^{\epsilon'}|^2\psi^{\epsilon'}\right)
\qtq{and}
F_2:=\left(\rho^{\epsilon}_{\bbS^1}-\rho^{\epsilon'}_{\bbS^1}\right)|\psi^{\epsilon'}|^2\psi^{\epsilon'}.
\]
Since the initial data cancels, we have
\[
\chi(\psi^{\epsilon}-\psi^{\epsilon'})=-i2\mu\int_0^t \chi e^{i(t-s)\Delta}\left(F_1(s)-F_2(s)\right)ds.
\]

First, we will show that
\begin{equation}\label{cauchy-nr}
\lim_{\epsilon,\epsilon'\to 0}\sup_{|t|\leq T}\|\chi(\psi^\epsilon-\psi^{\epsilon'})\|_{L^\infty_{r,+}L^4_\theta }=0
\end{equation}
We use our Littlewood--Paley projections, which gives us
\begin{align*}
    \sup_{|t|<T}\|\chi(\psi^\epsilon-\psi^{\epsilon'})\|_{L^\infty_{r,+}L^4_\theta }&\lesssim \sup_{|t|<T}\|P^r_{> N}E\chi(\psi^\epsilon-\psi^{\epsilon'})\|_{L^\infty_rL^4_\theta }\\&\quad +\sup_{|t|<T}\|P^\theta_{>N}P^r_{\leq N}E\chi(\psi^\epsilon-\psi^{\epsilon'})\|_{L^\infty_rL^4_\theta }+\sup_{|t|<T}\|P^\theta_{\leq N}P^r_{\leq N}E\chi(\psi^\epsilon-\psi^{\epsilon'})\|_{L^\infty_rL^4_\theta }.
\end{align*}
 For the high frequencies in \(r\), we use \eqref{bernie up r} and Sobolev embedding to bound
 \begin{align*}
     \sup_{|t|<T}\|P^r_{> N}E\chi(\psi^\epsilon-\psi^{\epsilon'})\|_{L^\infty_rL^4_\theta }\lesssim N^{-\frac 1 4}\sup_{|t|<T}\|P^r_{> N}E\chi(\psi^\epsilon-\psi^{\epsilon'})\|_{H^1(\R^2)}\lesssim N^{-\frac 1 4} A,
 \end{align*}
 where \(A\) is defined as in \eqref{oct 13}.

Similarly, using \eqref{bernie up theta} and \eqref{bernie down r}, we have
\begin{align*}
     \sup_{|t|<T}\|P^\theta_{>N}P^r_{\leq N}E\chi(\psi^\epsilon-\psi^{\epsilon'})\|_{L^\infty_rL^4_\theta }\lesssim N^{-\frac 1 4}\sup_{|t|<T}\|P^\theta _{> N}P^r_{\leq N}E\chi(\psi^\epsilon-\psi^{\epsilon'})\|_{H^1(\R^2)}\lesssim N^{-\frac 1 4} A.
 \end{align*}

For \(P^\theta_{\leq N}P^r_{\leq N}\), we decompose
\begin{align*}
    \sup_{|t|<T}\|P^\theta_{\leq N}P^r_{\leq N}E\chi(\psi^\epsilon-\psi^{\epsilon'})\|_{L^\infty_rL^4_\theta }&\lesssim\underbrace{\sup_{|t|\leq T}\|P^\theta_{\leq N}P^r_{\leq N}\int_0^tE\chi e^{i(t-s)\Delta}F_1(s)ds\|_{L^\infty_rL^4_\theta }}_{=:I}\\&\quad +\underbrace{\sup_{|t|\leq T}\|P^\theta_{\leq N}P^r_{\leq N}\int_0^tE\chi e^{i(t-s)\Delta}F_2(s)ds\|_{L^\infty_rL^4_\theta }}_{=:II}
\end{align*}

For I, we use that \(F_1 = \chi F_1\) and apply \eqref{dispersive-S-chi-C} to get
\begin{align*}
    \sup_{|t|\leq T}\|P^\theta_{\leq N}P_{\leq N}^r\int_0^tE\chi e^{i(t-s)\Delta}F_1(s)ds\|_{L^\infty_rL^4_\theta }&\lesssim \sup_{|t|\leq T}\int_0^t |t-s|^{-\frac 5 6}\|F_1(s)\|_{ L^1_{r,+}L^{\frac 4 3}_\theta}ds\\
    &\lesssim T^{\frac 1 6}\sup_{|t|\leq T}\| \rho^{\epsilon}\left(|\psi^{\epsilon}|^2\psi^{\epsilon}-|\psi^{\epsilon'}|^2\psi^{\epsilon'}\right)\|_{L^1_{r,+}L^{\frac 4 3}_\theta }\\
    &\lesssim T^\frac 1 6 A^2\sup_{|t|\leq T}\|\chi(\psi^\epsilon-\psi^{\epsilon'})\|_{L^\infty_{r,+}L^4_\theta }.
\end{align*}

For II,  taking \(f\in C_c^\infty(0,\infty)\) and recalling the definition \eqref{mollification}, for sufficiently small \(0<\epsilon,\epsilon'\leq\frac14\), we have 
\begin{align*}
\left|
\left\langle f,\rho^\epsilon-\rho^{\epsilon'}
\right\rangle_{L^2_{r,+}}
\right|&=\left|\int_{\mathbb R}
E(\rho^\epsilon-\rho^{\epsilon'})(r)
\,\overline{Ef(r)}\,dr\right|\\
&=\frac1{2\pi}\left|
\int_{\bbS^1}\int_{\mathbb R^2}\rho(z)
\left[
\frac{\overline{Ef(|y+\epsilon z|)}}{|y+\epsilon z|}
-
\frac{\overline{Ef(|y+\epsilon'z|)}}{|y+\epsilon'z|}
\right]dz\,dS(y)\right|\\
&\lesssim
\int_{\bbS^1}\int_{\mathbb R^2}\rho(z)
\left|Ef(|y+\epsilon z|)-Ef(|y+\epsilon'z|)\right|
\,dz\,dS(y) +
|\epsilon-\epsilon'|\,\|Ef\|_{L^\infty}
\int_{\mathbb R^2}\rho(z)|z|\,dz\\
&\lesssim_\rho
|\epsilon-\epsilon'|^{1/4}\|Ef\|_{H^{\frac{3}{4}}}\\
&\lesssim_\rho
|\epsilon-\epsilon'|^{1/4}\|f\|_{H^{\frac{3}{4}}_{r,+}},
\end{align*}
where the third line follows from the reverse triangle inequality \(\big||y+\epsilon z|-|y+\epsilon'z|\big|
\leq|\epsilon-\epsilon'|\,|z|\), and the penultimate line from the embedding \(H^{\frac{3}{4}}(\R)\hookrightarrow C^{0,\frac14}(\R)\). By duality, we arrive at the estimate
\begin{equation}\label{phodiffS}
\|\rho^\epsilon-\rho^{\epsilon'}\|_{H^{-\frac 3 4}_{r,+}}
\lesssim |\epsilon-\epsilon'|^{\frac 14}.
\end{equation}

We now apply this with \eqref{bernie down theta}, and Lemmas~\ref{lem:fakealgebraic} and~\ref{mixed-full}, with the fact that the linear Schr\"odinger propagator is isometry to conclude that
\begin{align*}
    \sup_{|t|\leq T}\|P^\theta_{\leq N}P_{\leq N}^r\int_0^t E\chi e^{i(t-s)\Delta}F_2(s)ds\|_{L^{4}_\theta H^{-\frac 3 4}_r}
    &\lesssim N^{\frac 1 4}\sup_{|t|\leq T}\left|\int_0^t\| P^\theta_{\leq N}E\chi e^{i(t-s)\Delta}F_2(s)\|_{L^{2}_\theta H^{-\frac 3 4}_r}\,ds\right|\\
    &\lesssim N\sup_{|t|\leq T}\left|\int_0^t\|P^\theta_{\leq N}\chi e^{i(t-s)\Delta}F_2(s)\|_{H^{-\frac 3 4}(\R^2)}\,ds\right|\\
    &\lesssim N T\sup_{|t|\leq T}\|P^\theta _{\leq N}F_2\|_{H^{-\frac 3 4}(\R^2)}\\
    &\lesssim N T\sup_{|t|\leq T}\|P^\theta _{\leq N}F_2\|_{L^{2}_\theta H^{-\frac 3 4}_{r,+}}\\
    &\lesssim N^{\frac 5 4}T\sup_{|t|\leq T}\|F_2\|_{L_\theta^{\frac 4 3}H^{-\frac 3 4}_{r,+}}\\
    &= N^{\frac 5 4}T\sup_{|t|\leq T}\|\left(\rho^{\epsilon}-\rho^{\epsilon'}\right)|\psi^{\epsilon'}|^2\psi^{\epsilon'}\|_{L^{\frac 4 3}_\theta H^{-\frac 3 4}_{r,+} }\\
    &\lesssim  N^{\frac 5 4}T\sup_{|t|\leq T}\|\rho^\epsilon-\rho^{\epsilon'}\|_{H^{-\frac 34}_{r,+}}\|\psi^{\epsilon'}\|^3_{H^1(\R^2)}\\
    &\lesssim  N^{\frac 5 4}T|\epsilon-\epsilon'|^\frac 14 A^3.
\end{align*}
Applying \eqref{bernie down r} again, we have
\[
\sup_{|t|\leq T}\|P^\theta_{\leq N}P_{\leq N}^r\int_0^tE\chi e^{i(t-s)\Delta}F_2(s)ds\|_{L^\infty_rL^4_\theta }\lesssim N^{\frac 5 2}T|\epsilon-\epsilon'|^\frac 14 A^3
\]

 Combining these estimates, we get
 \[
  \sup_{|t|<T}\|\chi(\psi^\epsilon-\psi^{\epsilon'})\|_{L^\infty_rL^4_\theta }\lesssim  N^{-\frac 1 4} A+T^\frac 1 6 A^2\sup_{|t|\leq T}\|\chi(\psi^\epsilon-\psi^{\epsilon'})\|_{L^\infty_{r,+}L^4_\theta }+N^{\frac 52}T|\epsilon-\epsilon'|^\frac 14 A^3.
 \]
 Choosing \(T\) sufficiently small to ensure that 
 \[
 TA^{12}\ll 1,
 \]
and taking \(\epsilon,\epsilon'\to 0\), we have
 \[
\limsup_{\epsilon,\epsilon'\to 0}\sup_{|t|\leq T}\|\chi(\psi^\epsilon-\psi^{\epsilon'})\|_{L^\infty_{r,+}L^4_\theta }\lesssim N^{-\frac 1 4} A.
 \]
Taking the limit as \(N\to\infty\), we see that \(\chi\psi^\epsilon\) converges in \(C([-T,T];L^\infty_{r,+}L^4_\theta )\).

Next, we will show that
 \begin{equation}\label{Cauchy H1}
      \limsup_{\epsilon,\epsilon' \to 0}\sup_{|t|\leq T}\|\psi^\epsilon-\psi^{\epsilon'}\|_{H^{-1}(\R^2)}=0.
 \end{equation}

 Using the previous decomposition, we have
 \begin{align*}
   \sup_{|t|\leq T}\|\psi^\epsilon-\psi^{\epsilon'}\|_{H^{-1}(\R^2)}&\lesssim  \sup_{|t|\leq T}\|\int_0^t e^{i(t-s)\Delta}F_1(s)ds\|_{H^{-1}(\R^2)}+ \sup_{|t|\leq T}\|\int_0^t e^{i(t-s)\Delta}F_2(s)ds\|_{H^{-1}(\R^2)}\\
   &\lesssim  \sup_{|t|\leq T}\|\int_0^t e^{i(t-s)\Delta}F_1(s)ds\|_{H^{-1}(\R^2)}+ T \sup_{|t|\leq T}\|F_2(t)\|_{L_\theta^{\frac43}H_{r,+}^{-\frac34}}\\
   &\lesssim  \underbrace{\sup_{|t|\leq T}\|\int_0^t e^{i(t-s)\Delta}F_1(s)ds\|_{H^{-1}(\R^2)}}_{:=I}+T|\epsilon-\epsilon'|^\frac 14 A^3,
 \end{align*}
 where the third inequality follows from Sobolev embedding. We then use \eqref{eq:anisotrophic-r-dual} and that \(F_1 = \chi^3 F_1\) to estimate
\begin{align*}
            I&\lesssim T\sup_{|t|\leq T}\|F_1(s)\|_{H^{-1}(\R^2)}\\
            &\lesssim T\sup_{|t|\leq T}\|\chi^3 F_1(s)\|_{ L^1_{r,+}L^{4/3}_\theta}\\
            &\lesssim T\|\rho_{\epsilon}\|_{L^1_{r,+}}\left(\sup_{|t|\leq T}\|\chi\psi^{\epsilon}(t)\|_{L^\infty_{r,+}L^4_\theta }^2+\sup_{|t|\leq T}\|\chi\psi^{\epsilon'}(t)\|_{L^\infty_{r,+}L^4_\theta }^2\right)\sup_{|t|\leq T}\|\chi\psi^{\epsilon}(t)-\chi\psi^{\epsilon'}(t)\|_{L^\infty_{r,+}L^4_\theta }\\
            &\lesssim TA^2\sup_{|t|\leq T}\|\chi\psi^{\epsilon}(t)-\chi\psi^{\epsilon'}(t)\|_{L^\infty_{r,+}L^4_\theta }.
        \end{align*}
It follows that
\[
\sup_{|t|\leq T}\|\psi^\epsilon-\psi^{\epsilon'}\|_{H^{-1}(\R^2)}\lesssim TA^2\sup_{|t|\leq T}\|\chi\psi^{\epsilon}(t)-\chi\psi^{\epsilon'}(t)\|_{L^\infty_{r,+}L^4_\theta }+T|\epsilon-\epsilon'|^\frac 14 A^3
\]
Passing the limit \(\epsilon,\epsilon'\to 0\) then proves that \(\psi^\epsilon\to\psi\) in \(C([-T,T];H^{-1}(\R^2))\).

Arguing precisely as in the proof of Proposition~\ref{p:10-15}, we may combine this convergence with the uniform \(H^1\)-bound to conclude that \(\psi^\epsilon\to \psi\) in \(C([-T,T];H_w^1(\R^2))\).  Moreover, an essentially identical argument to \eqref{identification} shows that \(\psi\) solves the Duhamel formulation \eqref{eq:Dulm} of \eqref{nls-sigma} with \(\Sigma = \bbS^1\) and is therefore unique by Proposition~\ref{p:10-15-2}.

         To complete the proof, we continue to argue as in Proposition~\ref{p:10-15} to prove norm convergence.  As before, it suffices to prove \eqref{pot energy converges}, but now with \(\Sigma=\bbS^1\).
       
        Proceeding as in \eqref{ND0}, we decompose
        \begin{equation*}
            \begin{aligned}
            \left|V_\epsilon(\psi^\epsilon(t))-V_0(\psi(t))\right|\leq
            \left| V_\epsilon(\psi^\epsilon(t))- V_0(\psi^\epsilon(t)) \right| 
            +\left|V_0(\psi^\epsilon(t)) - V_0(\psi(t))\right|.
            \end{aligned}
        \end{equation*}
        The first term once again vanishes uniformly for \(|t|\leq T\) as \(\epsilon \to 0\) by Lemma~\ref{lem:Aepsilon}. For the second term, we use that \(\chi\equiv 1\) near \(r=1\) to bound
        \[
            \begin{aligned}
            \left|V_0(\psi^\epsilon(t))-V_0(\psi(t))\right|&\lesssim
            \left( \|\psi^\epsilon(t,1,\cdot)\|_{L_\theta^4}^3+\|\psi(t,1,\cdot)\|_{L_\theta^4}^3\right)\|\psi^\epsilon(t,1,\cdot)-\psi(t,1,\cdot)\|_{L_\theta^4}\\
            &\lesssim A^3\|\chi(\psi^\epsilon-\psi)(t)\|_{L_{r,+}^\infty L_\theta^4},
            \end{aligned}
        \]
 
    which converges to zero uniformly for \(|t|\leq T\) by \eqref{cauchy-nr}. The proof of norm convergence is now completed using an identical argument to Proposition~\ref{p:10-15}, which allows us to upgrade convergence in \(C([-T,T];H_w^1(\R^2))\) to convergence in \(C([-T,T];H^1(\R^2))\) by Radon--Riesz, completing the proof.

\end{proof}
Once again, an identical bootstrap argument to Lemma~\ref{lem:uniform-regularized-control} proves the following:
\begin{cor}
    Let  \(n=2\), there exists a unique strong global solution \(\psi\in C(\R;\,H^1(\R^2))\) to \eqref{nls-S} with initial data $\psi_0\in H^1(\R^2)$ if either of the following two conditions holds:

(i)  \eqref{nls-y} is defocusing;

(ii) \(\|\psi_0\|_{H^1(\R^2)}\) is sufficiently small.

This solution conserves both the mass and the energy.

\end{cor}

An inspection of the proof of Proposition~\ref{Ctsdependence} demonstrates that it applies without modification to \(\Sigma = \bbS^1\). This yields continuity of the solution map and completes the proof of Theorem~\ref{thrm:GWPy}:
    
\begin{prop}
If \(\psi_{n,0}\to\psi_0\) is a convergent sequence of \(H^1(\R^2)\) initial data with corresponding solutions \(\psi_n,\psi\in C(\R;H^1(\R^2))\), then for any time \(0<T<\infty\) we have
\[
\lim_{n\to\infty}\sup_{|t|\leq T}\|\psi_n(t)-\psi(t)\|_{H^1(\R^2)}=0
\]

\end{prop}

\section{Local Wellposedness for \eqref{nls-y} on Higher Dimension}

In this section, we prove Theorem~\ref{thrm:LWP}, establishing local well-posedness for \eqref{nls-y} in high dimension.

Our proof of local well-posedness will rely on the following anisotropic Strichartz estimate:

\begin{lem}[Strichartz estimate]\label{thrm:stri}For any $\psi_0 \in L_{y}^2H_{x'}^s$ and any forcing term \(f\in L^{4/3}_t \M_{y}H_{x'}^s\)
\begin{equation}\label{Duhamel}
\psi(t) : = e^{-it\Delta }\psi_ 0 - i2\mu \int_{0}^t e^{i(t-s)\Delta }f(s)ds 
\end{equation}
satisfies the Strichartz estimate
    \begin{equation}\label{eq:Stri}
        \|\psi\|_{C_tL_{y}^2H_{x'}^s\cap L_t^4C_{y}^0H_{x'}^s}\lesssim \|\psi_0\|_{L_{y}^2H_{x'}^s} + \|f\|_{L^{4/3}_t \M_{y}H_{x'}^s}
    \end{equation}
\end{lem}
\begin{proof}We prove \eqref{eq:Stri} in the case \(s = 0\). The general case is obtained by applying the \(s=0\) case to \((1 - \Delta_{x'})^{\frac s2}\psi\).

From \cite{harropgriffiths2025scatteringnonlinearschrodingerequation}, we have the \(1d\) estimate,
    \begin{equation*}
        \|e^{it\p_y^2}\psi_0\|_{C_tL^2_y\cap L_t^4C_{0,y}}\lesssim \|\psi_0\|_{L^2_y}.
    \end{equation*}
    Applying the Fourier transform in \(x'\), we have
    \begin{equation}
    \begin{aligned}\label{111}
        \|e^{it\Delta}\psi_0\|_{C_tL^2_{y}L^2_{x'}\cap L_t^4C_{y}^0L^2_{x'}}&=\|e^{-it\sum_{i=1}^{n-1}\xi_i^2}e^{it\p_{y}^2}\hat{\psi_0}(\xi,{y})\|_{C_t L^2_{y}L^2_{\xi'}\cap L_t^4C_{y}^0L^2_{\xi'}}\\
        &\lesssim\|\hat{\psi_0}\|_{L^2_{\xi'}L^2_{y}}\\
        &=\|{\psi_0}\|_{L^2(\R^n)}.
    \end{aligned}
    \end{equation}
 
    From the fact that \(\M\) is the dual space of \(C_0\), we deduce the dual estimate,
    \begin{equation}\label{222}
      \|\int_\R e^{-it\Delta}f(t)dt\|_{L^2(\R^n)}\lesssim  \|f\|_{L^{4/3}_t\M_{y}L^2_{x'}}  
    \end{equation}
    Combining \eqref{111} and \eqref{222} with the Christ--Kiselev Lemma \cite[Theorem 1.2]{MR1809116}, we conclude the desired estimate
    \[
    \|\psi\|_{C_tL_{y}^2L^2_{x'}}+\|\psi\|_{L_t^4C_{0,y}L^2_{x'}}\lesssim \|\psi_0\|_{L^2(\R^n)} + \|f\|_{L^{4/3}_t\M_{y} L^2_{x'}}.
    \]
    
\end{proof}
\begin{prop}[Existence and Uniqueness]\label{LWP}
    Let \(s>\frac n 2\), then given \(\psi_0\in L_y^2H_{x'}^s\) there exists a time \(T=T(\psi_0)>0\) so that the equation \eqref{nls-y} admits a unique solution  \(\psi\in C_T L_y^2 H_{x'}^s\cap L^4_T C_{0,y}H_{x'}^s\) in both the focusing and defocusing cases. 

\end{prop}

\begin{proof}

    Given \(\eta>0\) and \(\psi_0\in L^2_{y}H^s_{x'}\), we may combine \eqref{thrm:stri} with Monotone Convergence Theorem to find \(T=T(\eta,\psi_0)>0\) such that
    \begin{equation}
        \|e^{it\Delta}\psi_0\|_{L^4_{[-T,T]}C_{y}^0H^s_{x'}}\leq \eta
    \end{equation}
    We then define the closed ball
\[
B=\{\psi\in L^4_{[-T,T]}C_{0,y}H^s_{x'}:\|\psi\|_{L^4_{[-T,T]}C_{0,y}H^s_{x'}}\leq 2\eta\}
\]
and consider the mapping
\[
\Phi[\psi](t):=e^{it\Delta}\psi_0-i\int_0^te^{i(t-s)\Delta}|\delta_{\{y=0\}}|\psi(s)|^2\psi(s)|dx
\]

Applying the Strichartz estimate \eqref{eq:Stri} and the algebra estimate \eqref{eq:algebraic}, for \(\psi\in B\) we may estimate
\begin{align*}
    \|\Phi[\psi]\|_{L^4_{[-T,T]}C_{y}^0H^s_{x'}}
    &\leq  \|e^{it\Delta}\psi_0\|_{L^4_{[-T,T]}C_{y}^0H^s_{x'}}+A\|\delta_{\{y=0\}}|\psi|^2\psi\|_{L^{4/3}_{[-T,T]} \M_{y}H^s_{x'}}\\
    &\leq  \eta+A\||\psi|^2\psi\|_{L^{4/3}_{[-T,T]} C_{y}^0H^s_{x'}}\\
    &\leq  \eta+A\|\psi\|^3_{L^{4}_{[-T,T]} C_{y}^0H^s_{x'}}\\
    &\leq \eta+8A\eta^3
\end{align*}
for some absolute constant \(A>0\) given by the Strichartz estimate \eqref{eq:Stri}. Similarly, for \(\psi,\phi\in B\) we may bound the difference 
\[
\|\Phi[\psi]-\Phi[\phi]\|_{L^4_{[-T,T]}C_{y}^0H^s_{x'}}\leq 24A\eta^2\|\psi-\phi\|_{L^4_{[-T,T]}C_{y}^0H^s_{x'}}
\]
For \(\eta>0\) sufficiently small, these estimates show that \(\Psi:B\to B\) and that this mapping is a contraction. Thus, there exists a unique \(\psi\) such that \(\Phi[\psi]=\psi\). Applying \eqref{eq:Stri} once again, we have
\begin{equation}\label{123}
    \begin{aligned}
    \|\psi\|_{C_{[-T,T]}L^2_{y}H^s_{x'}\cap L_{[-T,T]}^4C_yH_x^s}&\lesssim \|\psi_0\|_{L_{y}^2H^s_{x'}} + \|\delta_{\{y=0\}}|\psi|^2\psi\|_{L^{4/3}_{[-T,T]} \M_{y}H^s_{x'}}\\
    &\lesssim \|\psi_0\|_{L_{y}^2H^s_{x'}} +\||\psi|^2\psi\|_{L^{4/3}_{[-T,T]} C_{y}^0H^s_{x'}}\\
    &\leq\|\psi_0\|_{L_{y}^2H^s_{x'}} +\|\psi\|^3_{L^{4}_{[-T,T]} C_{y}^0H^s_{x'}}<\infty
    \end{aligned}
\end{equation}
and so \(\psi\in C_{[-T,T]}L^2_{y}H^s_{x'}\cap L_{[-T,T]}^4C_{y}^0H^s_{x'}\) is a solution to \eqref{nls-y} satisfying \(\psi(0)=\psi_0\).

\end{proof}

\begin{prop}[Continuous dependence on data]
    For any $\psi_0 \in L_y^2H_{x'}^s$ with solution on $[-T,T]$ where \(T=T(\psi_0)\) is chosen as in Proposition \ref{LWP} and any $0<\epsilon<T$, there exists $\delta > 0$ such that
    if 
    \[\|\tilde \psi_0  -\psi_0\|_{L_y^2H_{x'}^s} < \delta,\]
    then for \(T'=T-\epsilon\) we have
    \eq{Lipschitz}{
    \|\tilde \psi-\psi\|_{C_{[-T',T']}L_y^2H_{x'}^s \cap\ L^4_{[-T',T']} C_{0,y}H_{x'}^s}\lesssim \|\tilde \psi_0  -\psi_0\|_{L_y^2H_{x'}^s}.
    }
\end{prop}
\begin{proof}
    Taking \(\eta>0\) as in Proposition~\ref{LWP}, we recall that 
    \[\|e^{it\Delta}\psi_0\|_{L^4_{[-T,T]}C_{y}^0H^s_{x'}}\leq\eta.\]
By applying \eqref{eq:Stri}, we have
\begin{align*}
    \|e^{it\Delta}\tilde{\psi}_0\|_{L^4_{[-T',T']}C_{y}^0H^s_{x'}}&\leq \|e^{it\Delta}{\psi}_0\|_{L^4_{[-T',T']}C_{y}^0H^s_{x'}}+\|e^{it\Delta}(\tilde{\psi}_0-\psi_0)\|_{L^4_{[-T',T']}C_{y}^0H^s_{x'}}\\
    &\leq \eta+\widetilde A\delta
\end{align*}
for some absolute constant \(\widetilde A>0\). Taking \(\delta=\frac{\eta}{\widetilde A}\), we conclude that
\[
\|e^{it\Delta}\tilde{\psi}_0\|_{L^4_{[-T',T']}C_{y}^0H^s_{x'}}\leq 2\eta.
\]

Proceeding as in the proof of Proposition~\ref{LWP}, we now use \eqref{eq:Stri} to bound \begin{align*}
         \|\tilde\psi-\psi\|_{L_{[-T,T]}^4C_yH_x^s}&\leq \widetilde A\|\tilde\psi_0-\psi_{0}\|_{L_{y}^2H^s_{x'}} \\&\quad + \tfrac32A\|\tilde\psi-\psi\|_{L^{4}_{[-T,T]} C_{y}^0H^s_{x'}}(\|\tilde\psi\|_{L^{4}_{[-T',T']} C_{y}^0H^s_{x'}}^2+\|\psi\|_{L^{4}_{[-T',T']} C_{y}^0H^s_{x'}}^2)\\
         &\leq \widetilde A\|\tilde\psi_{0}-\psi_{0}\|_{L_{y}^2H^s_{x'}}+24A\eta^2\|\tilde\psi-\psi\|_{L^{4}_{[-T',T']} C_{y}^0H^s_{x'}}.
\end{align*}
Our choice of \(\eta\) ensures that \(24A\eta^2<1\), from which we obtain \eqref{Lipschitz}.
\end{proof}


\section{Radial Solutions for \eqref{nls-S} on Higher Dimension}
In this section, we prove Theorem~\ref{thrm:radial}, which demonstrates global well-posedness for \eqref{nls-S} with radial initial data when \(n\geq 3\). For this section, we assume \(f(x)=f(|x|)=f(r)\) for \(x\in \R^n\).

As in Section~\ref{sec:radial}, we use \(L_{r}^p\) and \(H_{r}^s\) to denote the \(1d\) Lebesgue and Sobolev spaces with \(r\in \R\). We take \(L_{r,+}^p\) and \(H_{r,+}^s\) to denote the corresponding \(1d\) Lebesgue and Sobolev spaces with \(r\in (0,\infty)\). Finally, we take \(L^p(r^{n-1}dr)\) and \(H^s(r^{n-1}dr)\) to denote the Lebesgue and Sobolev spaces for radial \(n\)-dimensional functions, with norms
\begin{alignat*}{3}
\|f\|_{L^p(r^{n-1}dr)}&:=\left(\int_0^\infty |f|^pr^{n-1}dr\right)^{\frac 1 p} \quad &&\text{for }1\leq p <\infty,\\
\|f\|_{H^s(r^{n-1}dr)}&:=\left(\int_0^\infty (1+\rho^2)^s|\tilde f(\rho)|^2\rho^{n-1} d\rho\right)^{\frac 1 2}\quad&&\text{for \(s\in \R\)},
\end{alignat*}
where we now take the Hankel transform of \(f\) to be
\[
\widetilde f(\rho)
:=\mathcal H_\nu[f](\rho)
:=\int_0^\infty
f(r)\frac{J_\nu(r\rho)}{(r\rho)^\nu}\,r^{n-1}\,dr,
\]
with \(\nu = \frac n2-1\).

We start by proving a weighted dispersive estimate that we will frequently use in this section.
\begin{lem}\label{dispersive_r_high}
Let \(\gamma=\frac n 2-\frac 2 3\), then
    \[
    \|r^{\gamma}e^{it\Delta}f\|_{C_{0,r,+}}\lesssim |t|^{-\frac 2 3}\|r^{-\gamma}f\|_{ L^1(r^{n-1}dr)}
    \]
\end{lem}
\begin{proof} Recalling that the kernel of \(e^{it\Delta}\) is \((4\pi it )^{-\frac n2}e^{i|x|^2/4t}\), we compute that
    \begin{align*}
        (e^{it\Delta} f)(r) &= \frac{1}{(4\pi i t)^{n/2}} e^{i\frac{r^2}{4t}} \int_0^\infty \left(\int_{\mathbb{S}^{n-1}} e^{-i\frac{r\rho}{2t} \omega \cdot \sigma} dS(\sigma) \,\right) e^{i\frac{\rho^2}{4t}} f(\rho) \rho^{n-1} d\rho\\
        &=\frac{1}{(4\pi i t)^{n/2}} e^{i\frac{r^2}{4t}} \int_0^\infty (2\pi)^{\frac{n}{2}} \left(\frac{r\rho}{2t}\right)^{-\nu} J_\nu\left(\frac{r\rho}{2t}\right) e^{i\frac{\rho^2}{4t}} f(\rho) \rho^{n-1} d\rho,
    \end{align*}
    where \(\nu=\frac n 2-1\) as above.  We again recall the Landau bound \cite{Landou2000Bessel} given by
    \[
    |J_\nu(x)|\lesssim |x|^{-\frac 1 3}, \quad\text{uniformly for all \(\nu\in \R^+\)}.
    \]
    It follows that
    \[
    |(e^{it\Delta} f)(r)|\lesssim |t|^{-\frac{n}{2}} \int_0^\infty \left(\frac{r\rho}{|t|}\right)^{-\nu - \frac{1}{3}} |f(\rho)| \rho^{n-1} d\rho,
    \]
    and therefore, noting that \(\gamma = \nu+\frac13\),
    \[
    r^{\gamma} |(e^{it\Delta} f)(r)| \lesssim |t|^{-\frac{2}{3}} \int_0^\infty \rho^{-\gamma} |f(\rho)| \rho^{n-1} d\rho.
    \]
    A similar argument to the proof of Lemma~\ref{l:dispersive estimate in section 4} demonstrates that the left hand side of this expression is continuous in \(r\) and vanishes at both zero and infinity.  Consequently,
    \[
    \|r^{\gamma}e^{it\Delta}f\|_{C_{0,r,+}}\lesssim |t|^{-\frac 2 3}\|r^{-\gamma}f\|_{ L^1(r^{n-1}dr)},
    \]
    as claimed.
\end{proof}

We now apply \eqref{dispersive_r_high} to \(\chi f\), where \(\chi(r)\) is again a smooth bump function supported on \((\frac 1 2, 2)\) that is identically \(1\) on \([\frac34,\frac54]\), to conclude that
\begin{equation}\label{dispersive-C-high}
    \|\chi e^{it\Delta}\chi f\|_{C_{0,r,+}}\lesssim |t|^{-\frac 2 3}\|\chi f\|_{ L^1(r^{n-1}dr)}\approx |t|^{-\frac 2 3}\|\chi f\|_{ L^1_{r,+}},
\end{equation}
    and by duality,
\begin{equation}\label{dispersive-M-high}
        \|\chi e^{it\Delta}\chi f\|_{L^\infty_{r,+}}\lesssim |t|^{-\frac 2 3}\|\chi f\|_{ \M(r^{n-1}dr)}\approx |t|^{-\frac 2 3}\|\chi f\|_{ \M_{r,+}}.
\end{equation}
We also note that by applying the \(1d\) Gagliardo--Nirenberg--Sobolev inequality and the product rule, if \(f\in H^1(\R^n)\) is radial then \(\chi f\in C_{0,r,}\) with the estimate
    \begin{align} \label{highd-embed}
        \|\chi f\|_{L^\infty_{r,+}}^2\lesssim \|\chi f\|_{L^2_{r,+}}\|\chi f\|_{\dot{H}^1_{r,+}}\lesssim\|f\|_{L^2_{([1/2,2];dr)}}\|f\|_{H^1_{([1/2,2];dr)}}\lesssim \|f\|_{H^1(\R^n)}^2.
    \end{align}

We use the same convention as Section~\ref{sec:radial} for the Littlewood--Paley projection \(P^r_{\geq N}\) in radial variables of a function $f(x)=f(r)$ with $r$ supported in a compact subset of $(0,\infty)$, i.e., we first extend by zero to \(r\in \R\), denoting the extension by \(Ef\), and then applying the usual Littlewood--Paley projection on \(\R\) in the variable \(r\). We state the corresponding Bernstein inequalities in the following lemma, with proof as in \cite{Tao2006NonlinearDispersive}:

\begin{lem}\label{Bernstein_high}
For all compactly supported \(f\colon (0,\infty)\to \C\) we have the estimates,
\begin{align}
    &\|P^r_{\leq N}Ef\|_{L^\infty_r}\lesssim N^{-s+\frac 1 2}\|Ef\|_{H^s_r}=N^{-s+\frac 1 2}\|f\|_{H^s_{r,+}} \qquad s<\frac 1 2,\label{bernie down high}\\
    &\|P^r_{> N}Ef\|_{L^\infty_r}\lesssim N^{-s+\frac 1 2}\|Ef\|_{H^s_r}=N^{-s+\frac 1 2}\|f\|_{H^s_{r,+}} \qquad s>\frac 1 2.\label{bernie up high}
\end{align}
\end{lem}

Proceeding as in the the case \(n=2\), we once again start with a uniqueness argument.
\begin{prop}\label{k-estimate-high}
    Suppose that for some \(T>0\) there are two radial solutions \(\psi,\phi\in C([-T,T];H_w^1(\R^n))\) of the Duhamel formulation \eqref{eq:Dulm} of \eqref{nls-S}. Then for all \(|t|\leq T\) we have the estimate
    \begin{equation}\label{eq:Gronwell-S2}
        \|\psi(t)-\phi(t)\|_{H^{-1}(\R^n)} \lesssim (1+K|t|e^{CK^{6}t^2}) \|\psi(0)-\phi(0)\|_{H^1(\R^n)},
    \end{equation}
    where \(C>0\) is some absolute constant and
    \[
    K:=\sup_{|t|\leq T}\Bigl\{\|\chi\psi(s)\|_{ L^\infty_{r,+}}^2+\|\chi\phi(s)\|^2_{ L^\infty_{r,+}}\Bigr\}\lesssim \sup_{|t|\leq T}\Bigl\{\|\psi(s)\|_{H^1}^2+\|\phi(s)\|^2_{H^1}\Bigr\}.
    \]
\end{prop}
\begin{proof} 
By the Duhamel formula, we have
    \[
    \psi(t)-\phi(t)=e^{it\Delta}(\psi(0)-\phi(0))-i2\mu\int_0^t e^{i(t-s)\Delta}\delta_{r=1}\left[|\psi(s)|^2\psi(s)-|\phi(s)|^2\phi(s)\right]ds,
    \]
    and may therefore bound
    \[
    \|\chi(\psi(t)-\phi(t))\|_{L^\infty_{r,+}}\lesssim \|\chi e^{it\Delta}(\psi(0)-\phi(0))\|_{L^\infty_{r,+}} +\left\|\int_0^t \chi e^{i(t-s)\Delta}\delta_{r=1}\left[|\psi(s)|^2\psi(s)-|\phi(s)|^2\phi(s)\right]ds\right\|_{L^\infty_{r,+}}.
    \]
    For the first term, we use \eqref{highd-embed} to bound
    \begin{align*}
        \|\chi e^{it\Delta}(\psi(0)-\phi(0))\|_{L^\infty_{r,+} }
        \lesssim \|e^{it\Delta}(\psi(0)-\phi(0))\|_{H^1(\R^n)}\lesssim \|\psi(0)-\phi(0)\|_{H^1(\R^n)}.
    \end{align*}

    For the second term, we apply our dispersive estimate \eqref{dispersive-M-high} to get
    \begin{align*}
        &\left\|\int_0^t \chi e^{i(t-s)\Delta}\delta_{r=1}\left[|\psi(s)|^2\psi(s)-|\phi(s)|^2\phi(s)\right]ds\right\|_{L^\infty_{r,+} }\\
        &\quad \lesssim  \int_0^t|t-s|^{-\frac 23}\|\chi \delta_{r=1}\left[|\psi(s)|^2\psi(s)-|\phi(s)|^2\phi(s)\right]\|_{\M_{r,+} }ds\\
        &\quad \lesssim K\int_0^t|t-s|^{-\frac 23}\|\chi(\psi(s)-\phi(s))\|_{L^\infty_{r,+}  }ds\\
        &\quad \lesssim K\left(\int_0^t|t-s|^{-\frac 4 5}ds\right)^\frac{5}{6}\left(\int_0^t\|\chi(\psi(s)-\phi(s))\|_{L^\infty_{r,+} }^6ds\right)^\frac 1 6\\
         &\quad \lesssim Kt^{\frac 1 6}\left(\int_0^t\|\chi(\psi(s)-\phi(s))\|_{L^\infty_{r,+}  }^6ds\right)^\frac 1 6.
    \end{align*}
     Taking the 6\textsuperscript{th} power gives us
     \[
     \|\chi(\psi(t)-\phi(t))\|_{L^\infty_{r,+} }^{6}\lesssim\|\psi(0)-\phi(0)\|_{H^1}^{6}+K^{6}t\int_0^t\|\chi(\psi(s)-\phi(s))\|^{6}_{L^\infty_{r,+}  }ds,
     \]
     and the integral form of Gronwall's inequality yields
    \eq{half way to uniqueness today}{
      \|\chi(\psi(t)-\phi(t))\|_{L^\infty_{r,+}  }\lesssim e^{CK^{6}t^2}\|\psi(0)-\phi(0)\|_{H^1(\R^n)}.
     }

     Finally, from~\eqref{highd-embed} we have
     \begin{align*}
         \langle\delta_{r=1}(|\psi|^2\psi-|\phi|^2\phi),f\rangle&=(|\psi(1)|^2\psi(1)-|\phi(1)|^2\phi(1))\overline{f(1)}\\&\lesssim (|\psi(1)|^2+|\phi(1)|^2)|\psi(1)-\phi(1)||f(1)|\\&\lesssim K\|\chi(\psi-\phi)\|_{L_{r,+}^\infty}\|f\|_{H^1(\R^n)},
     \end{align*}
     so by duality and \eqref{half way to uniqueness today}, we have
    \begin{align*}
        \|\psi(t)-\phi(t)\|_{H^{-1}(\R^n)}&\lesssim \|\psi(0)-\phi(0)\|_{H^{-1}(\R^n)}+\int_0^t\|\delta_{r=1}(|\psi|^2\psi-|\phi|^2\phi)\|_{H^{-1}(\R^n)}ds\\
        &\lesssim \|\psi(0)-\phi(0)\|_{H^{-1}(\R^n)}+K\int_0^t\|\chi(\psi-\phi)\|_{L_{r,+}^\infty}ds\\
        &\lesssim (1+Kte^{CK^{6}t^2}) \|\psi(0)-\phi(0)\|_{H^1(\R^n)}.\qedhere
    \end{align*}
\end{proof}

We now aim to follow parallel approximation argument to before.  Consequently, we first need to consider the problem of the existence of regularized solutions.  We start with the following nonlinear estimate:

\begin{lem}\label{product-highd-radial}
Let \(n\geq2\), and let
\(\rho^\epsilon_{\mathbb S^{n-1}}\) be defined by
\eqref{mollification}.
For \(0<\epsilon\leq1/4\), the fixed cutoff \(\chi\) satisfies
\[
\chi\equiv1
\quad\textrm{on }\operatorname{supp}
\rho^\epsilon_{\mathbb S^{n-1}}.
\]
Then, for every integer \(m\geq1\) and radial functions
\(f_1,\ldots,f_m\in H^1(\R^n)\),
\begin{equation}
\label{radial-product}
\left\|\rho^\epsilon_{\bbS^{n-1}}\prod_{j=1}^m f_j\right\|_{H^1(\R^n)}\lesssim\|\rho^\epsilon_{\bbS^{n-1}}\|_{H^1_{r,+}}\prod_{j=1}^m\|f_j\|_{H^1(\R^n)}.
\end{equation}
\end{lem}

\begin{proof}
For radial functions, \eqref{highd-embed}, and the fact that
\(r^{n-1}\approx_n1\) on \(\operatorname{supp}\chi\) give us
\[
\|\chi f_j\|_{L^\infty_{r,+}}
+
\|\chi\partial_r f_j\|_{L^2_{r,+}}
\lesssim_{n,\chi}
\|f_j\|_{H^1(\R^n)},
\qquad 1\leq j\leq m.
\]
Moreover, \eqref{highd-embed} also yields
\[
\|\rho^\epsilon_{\bbS^{n-1}}\|_{L^\infty_{r,+}}
\lesssim
\|\rho^\epsilon_{\bbS^{n-1}}\|_{H^1_{r,+}}.
\]

Since \(\chi\equiv1\) on the support of the \(\rho^\epsilon_{\bbS^{n-1}}\), the product rule and Hölder's inequality yield
\[
\begin{aligned}
\left\|\rho^\epsilon_{\bbS^{n-1}}\prod_{j=1}^m f_j
\right\|_{H^1(\R^n)}&\lesssim_n\left\|\rho^\epsilon_{\bbS^{n-1}}\prod_{j=1}^m f_j\right\|_{H^1_{r,+}}\\
&\lesssim\|\rho^\epsilon_{\bbS^{n-1}}\|_{H^1_{r,+}}\prod_{j=1}^m\|\chi f_j\|_{L^\infty_{r,+}}
+\|\rho^\epsilon_{\bbS^{n-1}}\|_{L^\infty_{r,+}}\sum_{j=1}^m\|\chi\partial_r f_j\|_{L^2_{r,+}}\prod_{\substack{1\leq k\leq m\\k\neq j}}\|\chi f_k\|_{L^\infty_{r,+}}\\
&\lesssim\|\rho^\epsilon_{\bbS^{n-1}}\|_{H^1_{r,+}}\prod_{j=1}^m\|f_j\|_{H^1(\R^n)},
\end{aligned}
\]
as required.
\end{proof}

For each fixed sufficiently small \(\epsilon>0\),
Lemma~\ref{product-highd-radial} implies that the nonlinearity in Duhamel formula is locally Lipschitz on the closed subspace of radial \(H^1(\R^n)\) functions. Standard semilinear well-posedness theory, based on the contraction
mapping principle, yields a unique maximal radial solution
\[
\psi^\epsilon\in C\bigl(I;H^1(\R^n)\bigr)
\]
of \eqref{nls-epsilon} for every radial initial datum
\(\psi_0\in H^1(\R^n)\), where \(I\) is an open interval containing zero; see \cite{cazenave2003semilinear}.

Taking \(V_\epsilon\) to be defined as in \eqref{V-epsilon definition} with \(\Sigma = \bbS^{n-1}\), we apply \eqref{highd-embed} to compute that
\[
\begin{aligned}\label{V4estimateshighd}
V_\epsilon(u)
&\leq
\|\chi u\|_{L^\infty_{r,+}}^4
\int_{\mathbb R^n}\rho^\epsilon_{\mathbb S^{n-1}}(x)\,dx\\
&=
|\mathbb S^{n-1}|\|\chi u\|_{L^\infty_{r,+}}^4
\\
&\lesssim \|u\|_{H^1(\mathbb R^n)}^4,
\end{aligned}
\]
uniformly for all \(0<\epsilon\leq \frac14\). This shows that we may apply Lemma~\ref{lem:uniform-regularized-control} to deduce that \(I = \R\) in the defocusing case, or the focusing case with small data, and obtain the uniform bound
\eq{oct 15}{
A: = \sup_{0<\epsilon\leq\frac14}\sup_{t\in\R}
\|\psi^\epsilon(t)\|_{H^1(\R^2)}<\infty.
}

We now prove the existence of solutions:

\begin{prop}
    Given radial \(\psi_0\in H^1(\R^n)\), there exists a time \(T=T(\|\psi_0\|_{H^1(\R^n)})>0\) and a solution \(\psi\in C([-T,T];H^1(\R^n))\) to the Duhamel formulation \eqref{eq:Dulm} of \eqref{nls-S} in the defocusing case, or in the focusing case when \(\|\psi_0\|_{H^1(\R^n)}\) is sufficiently small. Moreover, if \(\psi^\epsilon\) is a family of solutions to \eqref{nls-epsilon} with \(\psi^\epsilon(0)=\psi_0\in H^1(\R^n)\), then
    \[
    \lim_{\epsilon\to 0}\sup_{|t|\leq T}\|\psi^{\epsilon}(t)-\psi(t)\|_{H^1(\R^n)}=0
    \]
\end{prop}

\begin{proof}
    Taking \(0<\epsilon<\epsilon'\leq \frac14\), we use the Duhamel formulation \eqref{eq:Dulm} to write
    \[
    \psi^{\epsilon}-\psi^{\epsilon'}=e^{it\Delta}(\psi^\epsilon(0)-\psi^{\epsilon'}(0))- i2\mu\int_0^t e^{i(t-s)\Delta}\left(F_1(s)-F_2(s)\right)ds
    \]
    where
    \[
F_1:=\rho^\epsilon_{\bbS^{n-1}}\left(|\psi^{\epsilon}|^2\psi^{\epsilon}-|\psi^{\epsilon'}|^2\psi^{\epsilon'}\right)\qtq{and}
F_2:=\left(\rho^\epsilon_{\bbS^{n-1}}-\rho^{\epsilon'}_{\bbS^{n-1}}\right)|\psi^{\epsilon'}|^2\psi^{\epsilon'}.
\]
Since the initial data cancels, this simplifies to
\[
\chi(\psi^{\epsilon}-\psi^{\epsilon'})=-i2\mu\int_0^t \chi e^{i(t-s)\Delta}\left(F_1(s)-F_2(s)\right)ds.
\]

First, we will show that
\begin{equation}\label{cauchy-nr-highd}
\lim_{\epsilon,\epsilon'\to 0}\sup_{|t|\leq T}\|\chi(\psi^\epsilon-\psi^{\epsilon'})\|_{L^\infty_{r,+}}=0.
\end{equation}
We start by using our Littlewood--Paley projections to bound
\begin{align*}
    \sup_{|t|\leq T}\|\chi(\psi^\epsilon-\psi^{\epsilon'})\|_{L^\infty_{r,+} }&\lesssim \sup_{|t|\leq T}\|P^r_{> N}E\chi(\psi^\epsilon-\psi^{\epsilon'})\|_{L^\infty_r } +\sup_{|t|\leq T}\|P^r_{\leq N}E\chi(\psi^\epsilon-\psi^{\epsilon'})\|_{L^\infty_r}.
\end{align*}
 For the high frequencies in \(r\), we use \eqref{bernie up high} to estimate
 \begin{align*}
     \sup_{|t|\leq T}\|P^r_{> N}E\chi(\psi^\epsilon-\psi^{\epsilon'})\|_{L^\infty_r}\lesssim N^{-\frac 1 2}\sup_{|t|\leq T}\|P^r_{> N}\chi(\psi^\epsilon-\psi^{\epsilon'})\|_{H^1_r}\lesssim N^{-\frac 1 2} A,
 \end{align*}
 where \(A\) is defined as in \eqref{oct 15}. 

 For low frequencies, we have
\begin{align*}
    \sup_{|t|\leq T}\|P^r_{\leq N}E\chi(\psi^\epsilon-\psi^{\epsilon'})\|_{L^\infty_r }&\lesssim\underbrace{\sup_{|t|\leq T}\|P^r_{\leq N}\int_0^tE\chi e^{i(t-s)\Delta}F_1(s)ds\|_{L^\infty_r}}_{=:I}\\&\quad +\underbrace{\sup_{|t|\leq T}\|P^r_{\leq N}\int_0^tE\chi e^{i(t-s)\Delta}F_2(s)ds\|_{L^\infty_r }}_{=:II}.
\end{align*}

For I, we apply \eqref{dispersive-C-high} and \eqref{highd-embed} to bound
\begin{align*}
    \sup_{|t|\leq T}\|P_{\leq N}^r\int_0^tE\chi e^{i(t-s)\Delta}F_1(s)ds\|_{L^\infty_r }&\lesssim \sup_{|t|\leq T}\int_0^t |t-s|^{-\frac 2 3}\|\chi F_1(s)\|_{ L^1_{r,+}}ds\\
    &\lesssim T^{\frac 1 3}\sup_{|t|\leq T}\| \chi\rho^{\epsilon}_{\bbS^{n-1}}\left(|\psi^{\epsilon}|^2\psi^{\epsilon}-|\psi^{\epsilon'}|^2\psi^{\epsilon'}\right)\|_{L^1_{r,+} }\\
    &\lesssim T^\frac 1 3 A^2\sup_{|t|\leq T}\|\chi(\psi^\epsilon-\psi^{\epsilon'})\|_{L^\infty_{r,+} }.
\end{align*}

For II, from a parallel argument to \eqref{phodiffS} in Section \ref{sec:radial}, we conclude that

\[
\|\rho^\epsilon_{\bbS^{n-1}}-\rho^{\epsilon'}_{\bbS^{n-1}}\|_{H^{-\frac 34}_{r,+}}\lesssim |\epsilon-\epsilon'|^\frac 14.
\]
We then apply \eqref{eq:dualalgebraic1d} to conclude that
\begin{align*}
    \sup_{|t|\leq T}\|P_{\leq N}^r\int_0^tE\chi e^{i(t-s)\Delta}F_2(s)ds\|_{ H^{-\frac 3 4}_r}
    &\lesssim \sup_{|t|\leq T}\left|\int_0^t\|\chi e^{i(t-s)\Delta}F_2(s)\|_{ H^{-\frac 3 4}(\R^n)}ds\right|\\
    &\lesssim T\sup_{|t|\leq T}\|\chi\left(\rho^{\epsilon}_{\bbS^{n-1}}-\rho^{\epsilon'}_{\bbS^{n-1}}\right)|\psi^{\epsilon'}|^2\psi^{\epsilon'}\|_{ H^{-\frac 3 4}(\R^n)}\\
    &\lesssim  T\sup_{|t|\leq T}\|\rho^\epsilon_{\bbS^{n-1}}-\rho^{\epsilon'}_{\bbS^{n-1}}\|_{H^{-\frac 34}_{r,+}}\|\chi \psi^{\epsilon'}\|^3_{H^{\frac 3 4}_{r,+}}\\
    &\lesssim  T|\epsilon-\epsilon'|^\frac 14 A^3.
\end{align*}
Applying \eqref{bernie down high}, we then have
\[
\sup_{|t|\leq T}\|P_{\leq N}^r\int_0^tE\chi e^{i(t-s)\Delta}F_2(s)ds\|_{L^\infty_r }\lesssim  N^{\frac54}\sup_{|t|\leq T}\|P_{\leq N}^r\int_0^tE\chi e^{i(t-s)\Delta}F_2(s)ds\|_{ H^{-\frac 3 4}_r}\lesssim N^{\frac 54}T|\epsilon-\epsilon'|^\frac 14 A^3.
\]

 Combining these estimates, we get
 \[
  \sup_{|t|<T}\|\chi(\psi^\epsilon-\psi^{\epsilon'})\|_{L^\infty_r}\lesssim  N^{-\frac 1 2} A+T^\frac 1 3 A^2\sup_{|t|\leq T}\|\chi(\psi^\epsilon-\psi^{\epsilon'})\|_{L^\infty_r }+N^{\frac 54}T|\epsilon-\epsilon'|^\frac 14 A^3.
 \]

 Choosing \(T\) sufficiently small to ensure that
 \[
 TA^{6}\ll1,
 \]
and taking \(\epsilon,\epsilon'\to 0\), we have
\[
\limsup_{\epsilon,\epsilon'\to 0}\sup_{|t|\leq T}\|\chi(\psi^\epsilon-\psi^{\epsilon'})\|_{L^\infty_r}\lesssim N^{-\frac 1 2} A.
 \]
Finally, we take the limit as \(N\to\infty\), to show that \(\chi\psi^\epsilon\) converges in \(C([-T,T];L^\infty_{r,+})\).

Arguing precisely as in the proof of Propositions~\ref{p:10-15} and~\ref{GlobalS1}, we may combine this convergence with the uniform \(H^1\)-bound to conclude that \(\psi^\epsilon\to \psi\) in \(C([-T,T];H^{-1}(\R^n))\) and therefore \(C([-T,T];H_w^1(\R^n))\).  Moreover, an essentially identical argument to \eqref{identification} shows that \(\psi\) solves the Duhamel formulation \eqref{eq:Dulm} of \eqref{nls-sigma} with \(\Sigma = \bbS^{n-1}\) and is therefore unique by Proposition~\ref{p:10-15-2}.

To complete the proof, we continue to argue as in Proposition~\ref{p:10-15} and~\ref{GlobalS1} to prove norm convergence.  As before, it suffices to prove that the the potential energies converge as in \eqref{pot energy converges}.  We again decompose the potential energies as in \eqref{ND0}, to obtain
         \begin{equation}
            \begin{aligned}\label{ND_r}
            \left|V_\epsilon(\psi^\epsilon(t))-V_0(\psi(t))\right|\leq
            \left| V_\epsilon(\psi^\epsilon(t))- V_0(\psi^\epsilon(t)) \right| 
            +\left|V_0(\psi^\epsilon(t)) - V_0(\psi(t))\right|,
            \end{aligned}
        \end{equation}
        where \(V_\epsilon\) and \(V_{0}\) are defined as before, but now with $\Sigma = \bbS^{n-1}$.  We now proceed to bound each part and using parallel arguments to Proposition \ref{GlobalS1}.

        The main obstacle in higher dimensions is that the trace operator is only bounded from \(H^{\frac{n+1}4}(\R^n)\to L^4(\bbS^{n-1})\).  As we wish to rely on our uniform \(H^1\)-bounds, there is no hope of controlling the potential energy when \(n\geq 4\) without also using our hypothesis of radiality.  Fortunately, for radial solutions we can use the Sobolev estimate \eqref{highd-embed} in place of the trace estimate.  In dimension \(n=3\), the \(H^1\)-norm just suffices to control the trace; the preceding existence argument is the main obstruction to handling non-radial data in this case.

        For the first term on the right hand side of \eqref{ND_r}, proceeding as in Lemma \ref{lem:Aepsilon} it suffices to estimate
        \[
        \left| \int_\Sigma \left(|\tau_{\epsilon z}\psi^\epsilon|^4 - |\psi^\epsilon|^4\right) dS \right|\lesssim\left( \|\tau_{\epsilon z}\psi^\epsilon\|_{L^4(\Sigma)}^3 + \|\psi^\epsilon\|_{L^4(\Sigma)}^3 \right) \|\tau_{\epsilon z}\psi^\epsilon - \psi^\epsilon\|_{L^4(\Sigma)},
        \]
        where $\Sigma = \bbS^{n-1}$. Applying \eqref{highd-embed}, for \(|z|\leq 1\) and \(0<\epsilon \leq\frac14\) we have
        \[
        \|\tau_{\epsilon z}\psi^\epsilon\|_{L^4(\Sigma )}^4 = \int_{\Sigma} |\psi^\epsilon(y + \epsilon z)|^4 dS(y) 
        \lesssim \|\chi \psi^\epsilon\|_{L^\infty_{r,+}}^4 \lesssim \|\psi^\epsilon\|_{H^1(\R^n)}^4.
        \]

        Also note that,
        \[
        \|\tau_{\epsilon z}\psi^\epsilon - \psi^\epsilon\|^4_{L^4(\Sigma)} =\int_{\Sigma}|\psi^\epsilon(y+\epsilon z) - \psi^\epsilon(y)|^{4}dS(y),
        \]
        where the Fundamental Theorem of Calculus proves that if \(|z|\leq 1\) and \(0<\epsilon \leq\frac14\) then
        \begin{align*}
        \left|\psi^\epsilon(y+\epsilon z)-\psi^\epsilon(y)\right|
        &\leq
        \left|\,|y+\epsilon z|-|y|\,\right|^{1/2}
        \left\|\chi \partial_r \psi^\epsilon\right\|_{L_{r,+}^2} \\
        &\leq
        |\epsilon|^{1/2}|z|^{1/2}
        \left\|\chi \partial_r \psi^\epsilon\right\|_{L_{r,+}^2} \\
        & \lesssim  |\epsilon|^{1/2}\|\psi^\epsilon\|_{H^1(\R^n)}.
        \end{align*}
        Consequently,
        \[
        \begin{aligned}
        \left\|
        \tau_{\epsilon z}\psi^\epsilon(t)-\psi^\epsilon(t)
        \right\|_{L^4(\Sigma)}^4 \lesssim
        \epsilon^2
        \left\|\psi^\epsilon(t)\right\|_{H^1(\R^n)}^4.
        \end{aligned}
        \]
        Combining these estimates, we obtain
        \[
        \sup_{|t|\leq T}\left|V_\epsilon(\psi^\epsilon(t))-V_0(\psi^\epsilon(t))\right| \lesssim \epsilon^2A^4.
        \]

        For the second term, we simply bound
        \begin{align*}
        \left|V_0(\psi^\epsilon) - V_0(\psi)\right| &\lesssim \left( \|\psi^\epsilon\|_{L^4(\Sigma)}^3 + \|\psi\|_{L^4(\Sigma)}^3 \right) \|\psi^\epsilon - \psi\|_{L^4(\Sigma)}\lesssim A^3 \|\chi(\psi^\epsilon - \psi)\|_{L_{r,+}^\infty}.
        \end{align*}
        
        Taking $\epsilon \to 0$, we conclude that
        \[
        \lim_{\epsilon\to 0 }\sup_{|t|\leq T}\left|V_\epsilon(\psi^\epsilon(t))-V_{0}(\psi(t))\right| \to 0.
        \]

        The proof of norm convergence is now completed using an identical argument to Propositions~\ref{p:10-15} and~\ref{GlobalS1}, which allows us to upgrade convergence in \(C([-T,T];H_w^1(\R^n))\) to convergence in \(C([-T,T];H^1(\R^n))\) using Radon--Riesz, completing the proof.
\end{proof}

Combining our local existence and uniqueness with an identical boostrap argument to Lemma \ref{lem:uniform-regularized-control} proves the following:

\begin{cor}
    Let  \(n\geq 3\), there exists a unique strong global solution \(\psi\in C(\R;\,H^1(\R^n))\) to \eqref{nls-S} with initial data $\psi_0\in H^1(\R^n)$ if either of the following two conditions holds:

(i)  \eqref{nls-S} is defocusing;

(ii) \(\|\psi_0\|_{H^1(\R^n)}\) is sufficiently small.

This solution conserves both the mass and the energy.

\end{cor}

To complete the proof of Theorem~\ref{thrm:radial} in the case \(\Sigma = \bbS^{n-1}\), it remains to prove that the solution map is continuous:

\begin{prop}[Continuous dependence on initial data]
Let \(n\geq3\) and \(\Sigma=S^{n-1}\). Suppose that
\(\psi_{k,0},\psi_0\in H^1(\R^n)\) are radial and
\[
\psi_{k,0}\longrightarrow\psi_0
\qquad\text{in }H^1(\R^n).
\]
Let \(\psi_k,\psi\in C(\R;H^1(\R^n))\) be the corresponding radial
solutions of \eqref{nls-sigma}, in the defocusing case or in the
focusing case with all initial data sufficiently small.
Then, for every \(T>0\),
\[
\lim_{k\to\infty}
\sup_{|t|\leq T}
\|\psi_k(t)-\psi(t)\|_{H^1(\R^n)}
=0.
\]
\end{prop}
\begin{proof}
We follow the same argument as in Proposition~\ref{Ctsdependence}. Fix \(T>0\). Conservation of mass and energy, with the estimate \eqref{highd-embed} in place of the trace estimate, allows us to apply the argument of
Lemma~\ref{lem:uniform-regularized-control}, with \(V_\epsilon\)
replaced by \(V_0\) and \(\Sigma = \bbS^{n-1}\), to obtain a uniform \(H^1\)-bound  for \(\psi_k\) and \(\psi\).

The Lipschitz estimate~\eqref{eq:Gronwell-S2} yields
\[
\|\psi_k-\psi\|_{C_TH^{-1}(\R^n)}\to 0.
\]

In place of the trace inequality used in Proposition~\ref{Ctsdependence}, we now use that solutions are radial and apply the Gagliardo--Nirenberg inequality \eqref{1-d GNinfty}, interpolation, and \eqref{rdr=dr}, to obtain 
\[
\begin{aligned}
\|\psi_k-\psi\|_{C_TL^4(\Sigma)}\lesssim \|\chi(\psi_k - \psi)\|_{L_{r,+}^\infty}
&\lesssim
\|\chi(\psi_k-\psi)\|_{C_TH^{-1}_{r,+}}^{1/4}\|\chi(\psi_k-\psi)\|_{C_TH^1_{r,+}}^{3/4}\\
&\lesssim
\|\psi_k-\psi\|_{C_TH^{-1}(\R^n)}^{1/4}\|\psi_k-\psi\|_{C_TH^1(\R^n)}^{3/4}
\to 0\qt{as \(k\to\infty\)}.
\end{aligned}
\]

Together with convergence of the initial data and conservation of
mass and energy, this gives uniform convergence of the
\(H^1\)-norms.

The Hilbert-space identity and compactness argument
from Proposition~\ref{Ctsdependence} then apply without further
modification, yielding
\[
\lim_{k\to\infty}\sup_{|t|\leq T}\|\psi_k(t)-\psi(t)\|_{H^1(\R^n)}=0,
\]
which completes the proof.
\end{proof}

\bibliographystyle{abbrv}
\bibliography{refs}
\end{document}